\documentclass[12pt]{amsart}
\usepackage[margin=1in]{geometry}
\usepackage{amsmath,amsfonts,amsthm,amssymb,bbm}
\usepackage{graphicx,color,dsfont}
\usepackage{enumitem}
\usepackage{fourier}

\newtheorem{theorem}{Theorem}
\newtheorem{lemma}{Lemma}
\newtheorem{proposition}{Proposition}

\newtheorem{definition}{Definition}
\newtheorem{assumption}{Assumption}	
	\newtheorem{corollary}{Corollary}

\newtheorem{claim}{Claim}

 \theoremstyle{definition}
 
 \theoremstyle{remark}

 \numberwithin{equation}{section}

\newcommand{\vertiii}[1]{{\left\vert\kern-0.25ex\left\vert\kern-0.25ex\left\vert #1
    \right\vert\kern-0.25ex\right\vert\kern-0.25ex\right\vert}}

\newcommand{\f}[2]{\frac{#1}{#2}}

\newcommand{\cl}{{\mathcal L}}
\newcommand{\cu}{{\mathcal U}}

\newcommand{\al}{\alpha}
\newcommand{\be}{\beta}

\newcommand{\de}{\delta}

\newcommand{\ka}{\kappa}
\newcommand{\la}{\lambda}

\newcommand{\La}{\Lambda}
\newcommand{\si}{\sigma}

\newcommand{\vp}{\varphi}

\newcommand{\om}{\omega}

\newcommand{\rone}{\mathbb R}

\newcommand{\dpr}[2]{\langle #1,#2 \rangle}
\newcommand{\dprr}[2]{( #1,#2 )}

\newcommand{\eps}{\epsilon}

\newcommand{\ca}{\mathcal A}

\newcommand{\cm}{\mathcal M}

\newcommand{\ce}{\mathcal E}

\newcommand{\ch}{\mathcal H}

\newcommand{\cp}{\mathcal P}
\newcommand{\cx}{\mathcal X}

\newcommand{\p}{\partial}

\newcommand{\beq}{\begin{equation}}
\newcommand{\eeq}{\end{equation}}
\newcommand{\beqna}{\begin{eqnarray*}}
\newcommand{\eeqna}{\end{eqnarray*}}
\newcommand{\beqn}{\begin{equation*}}
\newcommand{\eeqn}{\end{equation*}}
\newcommand{\bp}{\begin{proof}}
\newcommand{\ep}{\end{proof}}
\newcommand{\bprop}{\begin{proposition}}
\newcommand{\eprop}{\end{proposition}}
\newcommand{\bt}{\begin{theorem}}
\newcommand{\et}{\end{theorem}}
\newcommand{\bex}{\begin{Example}}
\newcommand{\eex}{\end{Example}}
\newcommand{\bc}{\begin{corollary}}
\newcommand{\ec}{\end{corollary}}
\newcommand{\bcl}{\begin{claim}}
\newcommand{\ecl}{\end{claim}}
\newcommand{\bl}{\begin{lemma}}
\newcommand{\el}{\end{lemma}}

\newcommand{\cj}{{\mathcal J}}

\begin{document}

\title[Stability of   waves for the fractional KdV and NLS equations]
 {Stability of periodic   waves for the fractional KdV and NLS equations}

\thanks{Sevdzhan Hakkaev partially supported by Scientific Grant RD-08-119/2018 of Shumen University.
  Stefanov  is partially  supported by  NSF-DMS under grant \# 1614734.}

\author[S. Hakkaev]{\sc Sevdzhan Hakkaev}
\address{ Department of Mathematics and Computer Science, Istanbul Aydin University, Istanbul, Turkey}

\email{sevdzhanhakkaev@aydin.edu.tr}

\address{Faculty of Mathematics and Informatics, Shumen University, Shumen, Bulgaria}

\author[A. Stefanov]{\sc Atanas G. Stefanov}
\address{ Department of Mathematics,
University of Kansas,
1460 Jayhawk Boulevard,  Lawrence KS 66045--7523, USA}
\email{stefanov@ku.edu}



\date{\today}

\begin{abstract}
We consider the focusing fractional  periodic  Korteweg-deVries (fKdV) and fractional periodic nonlinear Schr\"odinger equations (fNLS)  equations, with $L^2$ sub-critical dispersion.  In particular, this covers the case of the periodic KdV and Benjamin-Ono models.  We construct  two parameter family of bell-shaped traveling waves for KdV (standing waves for NLS), which are constrained minimizers of  the Hamiltonian. We show in particular that for each $\la>0$, there is a traveling wave solution to fKdV and fNLS  $\phi: \|\phi\|_{L^2[-T,T]}^2=\la$, which is non-degenerate and spectrally stable, as well as orbitally stable. This is done completely rigorously, without any  {\it a priori} assumptions on the smoothness of the waves or the Lagrange multipliers.

\end{abstract}

\maketitle

\section{Introduction}
  Consider the initial value problem for the  fractional periodic KdV equation
  \begin{equation}
\label{10}
\left\{\begin{array}{l}
u_t- \La^\al u_x + (u^2)_x = 0, -T\leq x\leq T,\\
u(0,x)=u_0(x)
\end{array}
\right.
  \end{equation}
  and the corresponding quadratic NLS problem,
  \begin{equation}
  \label{20}
  \left\{\begin{array}{l}
  i u_t- \La^\al u + |u| u = 0, -T\leq x\leq T.\\
  u(0,x)=u_0(x)
  \end{array}
  \right.
  \end{equation}
  Here, the fractional differentiation operator is defined, say for finite trigonometric polynomials via
  $$
  \La^\al [\sum_{k=-N}^N   a_k e^{ i \pi k \f{x}{T}}] = \sum_{k=-N}^N  \left(\f{\pi |k|}{T}\right)^\al a_k e^{ i \pi k \f{x}{T}},
  $$
  and then by extensions to all elements of $H^{\al}[-T,T]$. 
  
  In both problems, we can in principle  consider any $\al>0$, although we will see that for meaningful results, one needs to restrict to $\al >\f{1}{3}$. The cases $\al=1$ and $\al=2$ are of course classical and well-studied - these are the Benjamin-Ono and the KdV models respectively. The local and global well-posedness theory has been well developed for these standard  cases, even for very low   regularity data,  see \cite{iteam} for KdV and \cite{mol, IK}  for the Benjamin-Ono case.  For non-integer $\al$, we mention the relatively recent works \cite{KMR, HIKK} for the fKdV posed on the line, which provides global well-posedness in the energy space $H^{\f{\al}{2}}(\rone)$. Note that the flow maps in the non-local cases, i.e.  $\al<2$,  are generally not even uniformly continuous with respect to initial data. The state of affairs regarding the Cauchy problem for the fNLS, \eqref{20}  is as follows:  the local well-posedness results is addressed in \cite{MPV}, for data in $H^s, s>\f{3}{2}-\f{5\al}{4}$, while global existence is in the energy space $H^{\f{\al}{2}}(\rone), \al>\f{6}{7}$.

The existence of traveling waves (standing waves respectively) is another  aspect of the theory, as it offers important information regarding  the global dynamical properties of these models. In fact, such solutions (and the behavior of the solutions starting with data close to them) provide  the most important clues and indeed the skeleton of the full dynamic picture. This is why the problem for the existence and stability properties of traveling waves for these and related equations has played such a central role. In that regards, we mention \cite{Ben}, where the non-linear stability of the KdV  traveling waves on the line  was established. This was followed by \cite{BSS},   and the far reaching generalizations in \cite{GSS, LZ}. A very satisfactory result, including uniqueness for the line soliton, holds   for Benjamin-Ono model as well,  \cite{AT}. Some other recent results for traveling waves on the line, for the non-local models $\al<2$, as well as more general multipliers are  in \cite{PSB, LPS1}.  See also the book \cite{Pbook}, where the approach for many of  these results can be found.

 In recent years, the periodic traveling waves for these and related models, together with their stability properties,   were considered in numerous papers. Here is a  list of some recent developments, \cite{PBS, Pava, PN, HJ, LP}, which is certainly incomplete\footnote{In addition, there is quite a bit of recent works dealing with instabilities of such waves. We do not review these issues here, as our results pertain exclusively to  stability.}.  In most of these works, the waves are constructed either   variationally or through some ODE based methods\footnote{although some of these waves are in fact explicit}. The orbital stability considerations for these waves often  involve  some variation of  the Benjamin's   method, \cite{Ben}. Note that an essential ingredient in this approach is the so-called {\it non-degeneracy of the wave}, which roughly states that  the kernel of the linearized operator $\cl_+$ , see \eqref{11} and \eqref{12}below,  is spanned by $\phi'$.  This is also an important issue, which arises, when one analyzes  the uniqueness of the waves in \eqref{30} as well.   We shall provide more details about the specifics of these works, as it pertains to our contribution below, see the discussion after  Theorem \ref{theo:main}.

Let us record three important {\it formally} conserved quantities for the solutions of \eqref{10}.
\begin{itemize}
	\item the $L^2$ norm/momentum
	$$
	\cp(u)=  \int_{-T}^T u^2(x) dx
	$$
	\item the Hamiltonian
	$$
	\ch(u)=\f{1}{2} \dpr{\La^{\al/2} u}{\La^{\al/2} u} - \f{1}{3} \int_{-T}^T u^3(x)   dx
	$$
	\item the mass
	$$
	\cm(u)=\int_{-T}^T u(x) dx
	$$
\end{itemize}
while clearly only $\cp, \ch$ are conserved on the solutions of \eqref{20}. {\it Let us note however, that even for the cases where one has global well-posedness, it is generally not clear whether these quantities are actually conserved along the evolution, especially when one works with spaces with limited regularity, say $H^{\f{\al}{2}}[-T,T]$. }

For traveling waves of \eqref{10}, we take the ansatz $u(t,x)=\phi(x- \om t)$, while for \eqref{20}, we consider $u(t,x)=e^{ i \om t} \phi(x)$. In addition, we will be interested in positive solutions $\phi$ only, In the case of \eqref{10}, we obtain, after one integration,  the profile equation
\begin{equation}
\label{30}
\La^\al \phi+\om \phi - \phi^2 +a=0, -T\leq x\leq T,
\end{equation}
  where $a\in\rone$ is a constant of integration. In the case of \eqref{20}, we obtain exact same equation, but with  $a=0$. Thus, we will generally consider \eqref{30} with any $a$, and sometimes we will refer specifically to the case $a=0$ as it concerns the NLS problem \eqref{20}.
  Another helpful reduction,   that we would like to point out   is the following scaling argument.
  More specifically, using the transformation $\phi(x)= T^{-\al} \Phi(x/T)$, where $\Phi$ is now $2$ periodic, leads us to the equation
 $$
  \La^\al \Phi+T^\al \om \Phi - \Phi^2 + T^{2\al}  a =0, -1\leq y\leq 1.
 $$
  Also, a moment thought reveals  that the stability of $\Phi$ is equivalent to the stability of $\phi$, in the context of \eqref{10} or \eqref{20}. So, understanding the stability of $\Phi$, as a function of the parameters $\om, a$ on the basic interval $[-1,1]$ can leads to the answer of the stability of waves defined on any interval\footnote{So, henceforth,  without loss of generality, we shall mostly restrict our attention to the case $T=1$. } $[-T,T]$.

  We now discuss the problem for stability of the traveling waves  $\phi_\om(x-\om t)$ for \eqref{10} and the standing waves $e^{i \om t} \phi_\om$ for \eqref{20}, provided they exist. For the fKdV case, taking the ansatz
  $u=\phi_\om(x-\om t)+v(t, x-\om t)$, and for the fNLS case, we take
  $u=e^{i \om t}[\phi_\om+ v_1(t,\cdot)+iv_2(t, \cdot)]$, where $v_1, v_2$ are taken to be real-valued.  Plugging in
  \eqref{10} (\eqref{20} respectively) and ignoring $O(|v|^2)$,  leads us to the linearized systems
  \begin{eqnarray}
  \label{11}
 & &   v_t =\p_x \cl_+ v, \\
  \label{12} 
 & &  \vec{v}_t =   \left(\begin{array}{cc}
  0 & 1  \\
  -1 & 0
  \end{array}\right)   \left(\begin{array}{cc}
  \cl_+ & 0  \\
  0 & \cl_-
  \end{array}\right) \vec{v},
  \end{eqnarray}
  where the linearized operators $\cl_\pm$, take the form
  \begin{eqnarray*}
\cl_+=\La^\al+\om - 2\phi, \ \ \cl_-= \La^\al+\om - \phi, \ \ D(\cl_\pm)=H^\al.
  \end{eqnarray*}
  Passing to the eigenvalue ansatz $v\to e^{\la t} v$ yields the relevant eigenvalue problems
   \begin{eqnarray}
   \label{111}
   & & 
   \p_x \cl_+ v=\la v, \\
   \label{121}
   & & 
    \left(\begin{array}{cc}
   0 & 1  \\
   -1 & 0
   \end{array}\right)  \left(\begin{array}{cc}
   \cl_+ & 0  \\
   0 & \cl_-
   \end{array}\right) \vec{v}=\la \vec{v}.
   \end{eqnarray}
  A straightforward comparison of $\cl_\pm$ with the constant coefficient operator $\cl^0=\La^\al+\om$ and the fact that the spectra  of the operators $\p_x \cl^0$ and $\left(\begin{array}{cc}
  0 & 1  \\
  -1 & 0
  \end{array}\right)  \left(\begin{array}{cc}
  \cl^0 & 0  \\
  0 & \cl^0
  \end{array}\right)$ consist entirely of eigenvalues of finite multiplicity implies, by Weyl's criteria,  that the spectral problems \eqref{111} and \eqref{121} have their entire spectrum filled with  eigenvalues of finite multiplicity. With that in mind, we give the following definitions of stability. 
  \begin{definition}
  	\label{defi:1} 
We say that the wave $\phi_\om(x-\om t)$ is spectrally stable, with respect to perturbations of the same period, if $\si(\p_x \cl_+)\subset \{\la: \Re \la\leq 0\}$. Alternatively,  the eigenvalue problem \eqref{111} does not have solutions $(\la, v): \Re \la>0, v\in D(\p_x \cl_+), v\neq 0$. 

For the fNLS problem, the spectral stability of $e^{i \om t} \phi_\om$ is  understood as the absence of non-trivial solutions of \eqref{121}, with $\Re \la>0$. 
  \end{definition}
  The stronger notion of orbital stability for fKdV will be useful in the sequel. As we have mentioned above, the results in this direction are conditional upon a  well-posedness results, in addition to actual conservation of the momentum $\cp(u)$ and the Hamiltonian $\ch(u)$. 
  \begin{assumption}
  	\label{defi:15} Let $\phi$ be a solution of \eqref{30}. Assume  that there exists $\eps_0>0$ and a metric space $(X, d_X)$, so that 
  	  $X\subset \{g\in H^{\f{\al}{2}}[-T,T]: d_X(g,\phi)<\eps_0\}$, with the following properties: 
  	  \begin{enumerate}
  	  	\item The solution map $g\to u_g$ is locally in time continuous in the metric $d_X$. That is, for each $\eps<\eps_0$ and $g: d_X(g, \phi)<\f{\eps}{2}$, there exists $t_0=t_0(g)>0$, so that 
  	  	$\sup_{0<t<t_0} d_X(u_g(t), \phi)<\eps$. 
  	  	\item All initial data $g\in X$ generates a global in time solution $u_g$ of  \eqref{10}, so that \\ $g\in C((0,\infty)), H^{\f{\al}{2}}[-T,T])$. 
  	  	\item For all  $0<t<\infty$, there is the conservation of momentum, Hamiltonian and mass 
  	  	$$
  	  	\cp(u_g(t, \cdot))=\cp(g), \ch(u_g(t, \cdot))=\ch(g), \cm(u_g(t, \cdot))=\cm(g),
  	  	$$
  	  \end{enumerate}
  \end{assumption}
  Loosely speaking, we require relatively strong well-posedness result to hold in a suitable subspace of $H^{\f{\al}{2}}[-T,T]$, so that the relevant conserved quantities $\cp, \ch, \cm$ are conserved along the evolution. For example, a global well-posedness in a space of sufficiently high regularity, say $H^3$,  would  be ideal, since then, the solutions to \eqref{10} are classical and the conservation laws calculations are justified. This holds for example, in the cases of KdV and Benjamin-Ono,  but one then is restricted to taking only perturbations $u_0$, which are sufficiently smooth. 
  One should also remember   that generally speaking, in the cases $\al<2$, the data-to-solution map is not uniform continuous in the scale of the Sobolev spaces (of any order!).  This idiosyncrasy of the model \eqref{10} is not consequential for our results, as we only focus on having global unique solutions, which necessarily conserve  $\cp, \ch, \cm$.  
  
  Our next definition is about the orbital stability of the waves. 
  \begin{definition}
  	\label{defi:17}
  	 We say that $\phi$ is  orbitally stable, if for every $\eps>0$,  
  	 there exists $\de>0$, so that whenever $u_0\in X$, $\|u_0-\phi\|_{H^\f{\al}{2}}<\de$ and $u_0$ is real-valued, then the solution $u$ is globally in $H^{\f{\al}{2}}[-T,T]$ and 
  	 $$
  	 \sup_{t>0} \inf_{y\in [-T,T]}  \|u(t,\cdot+y)-\phi(\cdot)\|_{H^\f{\al}{2}[-T,T]}<\eps.
  	 $$
  	 Similarly, orbital stability for fNLS is  for every $\eps>0$, there is $\de>0$, so that whenever the initial data is   	 $u_0\in X, \|u_0-\phi\|_{H^\f{\al}{2}}<\de$,   there is a global solution 
  	 $u(t, \cdot)$,  so that 
  	 $$
  	 \sup_{t>0} \inf_{y\in [-T,T], \theta\in [0,2\pi]}   \| u(t,\cdot+y)- e^{i \theta} \phi_\om(\cdot)\|_{H^\f{\al}{2}[-T,T]}<\eps.
  	 $$
  \end{definition}
   We are now ready to state the main results of this paper.
  \begin{theorem}
  	\label{theo:main}
  	Let $\al \in (\f{1}{2},2]$,  $T>0$. Then,   for each $\la>0$ and $a\in\rone$, there is a bell-shaped and classical  solution of \eqref{30}, $\phi_{\la,a} \in H^{\infty}[-T,T]$, where  $\om=\om(\la, a,  \phi_{\la,a})$ and 
  	$\int_{-T}^T \phi^2(x) dx=\la$. 
  	
  	In addition, the  corresponding traveling wave solutions $\phi_{\om_\la} (x-\om_\la t)$ of the fKdV equation, \eqref{10} are  non-degenerate, when $a\neq \f{\la}{2T}$ and spectrally stable, in the sense of Definition \ref{defi:1}. 
  	 Moreover, assuming global well-posedness as in Assumption \ref{defi:15}, the waves are also  orbitally stable,  when $a\neq \f{\la}{2T}$. 
  	
  	Similarly, for $a=0$,  the standing wave solutions $e^{i \om_\la t} \phi_{\om_\la}$ of the  fNLS, \eqref{20} are  non-degenerate  and spectrally stable. Under  Assumption  \ref{defi:15},   one can upgrade the statements to orbital stability. 
  \end{theorem}
  {\bf Remarks:} 
  \begin{itemize}
  	\item The maps $a\to \om(a), a\to \phi_a$ are certainly of interest (for example continuity, differentiability and monotonicity  properties etc.),  but they will not be a subject of our investigation, so we will henceforth drop it from our notation. 
  	\item The restriction $a\neq \f{\la}{2T}$ is likely an artifact of the argument, but we cannot remove it for now. 
  	\item It is somewhat implicit in the statement that the wave speed $\om$ may depend on the particular solution $\phi$. To clarify this important  point,  we cannot rule out a scenario where   for a given $(\la, a)\in \rone_+\times \rone$, there are two waves $\phi, \tilde{\phi}:\|\phi\|^2=\la=\|\tilde{\phi}\|^2$ satisfying \eqref{30}, with $\om(\la,a,\phi,)\neq \om(\la,a,\tilde{\phi})$.
  	\item In relation to the previous point, $\la\to \om(\la)$ may be a multi-valued mapping. On the other hand, in Proposition \ref{prop:11} below, we clarify that on a full measure subset $\ca\subset \rone_+$, $\om_\la$ is independent on the waves of our construction.  
  \end{itemize}
  
  We should mention that periodic  waves, in the fKdV context, were previously constructed in \cite{HJ}. In this work, the authors have used different variational construction, namely they construct the solutions  subject to the constraint $\int_{-T}^T \phi^3(x) dx= const.$, which is why they can get to the larger range $\al>\f{1}{3}$. In the stability arguments,  the authors tacitly assume smoothness of the Lagrange multipliers\footnote{while on a more basic level, and as was discussed above, it is not at all  clear why these multipliers are independent on the particular constrained minimizers} on the constraints, which simplifies matters quite a bit. Using the assumed smoothness, they show the orbital stability of the waves.

  Our approach does not make use of any such assumptions. In fact, let us give  an informal preview  to our existence and  stability results, together with the difficulties associated with various steps in the proof.  We construct first, for each $\la>0$, normalized waves, that is functions  that minimize the modified energy $\ch(u)+ a M(u)$ for fixed $L^2$ norm, $\|\phi\|_{L^2}^2=\la$, see Proposition \ref{prop:10} below.  This procedure generates  bell-shaped functions, with   speeds $\om_{a, \la, \phi}$ as Lagrange multipliers.  
  
  The  smoothness (or even continuity) of the map $(a, \la)\to \om_{a, \la}$ is a highly non-trivial issue. In fact, we show that $\la\to \om_{a, \la}$ is non-decreasing, while the continuity and differentiability of this map remains an open question.    Even more dramatically, the continuity, let alone the differentiability,  of the Banach space valued mapping $\la\to \phi_\la$ remains an open and very challenging  question. This is often an assumption, see  \cite{Pbook} and also $(5.2.47)$ on p. 139 in \cite{Kap} where this is explicitly required.  The  issue was  sidestepped as an obvious one in previous publications.  While we accept that the continuity and even differentiability is very likely true, 
  we would want to reiterate the fact that it  is not obvious, except in the cases with scaling (i.e. when the problem is posed on $\rone$, instead of $[-T,T]$), in which the relation $\om\to \phi_\om$ is explicit.  
  
  While we do not make any continuity/differentiability assumptions of the sort,  we certainly would benefit from such smoothness properties. In fact, we prove some very modest results along these lines, see  Proposition \ref{prop:11} and Lemma \ref{le:65} below, which however turn out to suffice  for our purposes.   For example, a  key step in the argument, is the {\it weak non-degeneracy} of $\phi$, i.e. $\phi\perp Ker[\cl_+]$.  Note that this is  trivial\footnote{Indeed, taking {\it formally} derivatives in $\om$ in \eqref{30} leads to $\cl_+[\p_\om \phi]=-\phi$, whence $\phi\perp Ker[\cl_+]$.  }, if one assumes the $H^1$ smoothness of the map  $\om\to \phi_\om$.    With the non-degeneracy at hand, one proceeds to establish  that the waves are non-degenerate, in the sense that $Ker[\cl_+]=span[\phi']$. This is then a crucial piece of information, which is needed in the proof of orbital stability for  these waves. 
  
  The   paper is organized as follows. In Section \ref{prelim}, we first show that the distributional solutions of \eqref{30} are in fact $H^\infty$. In Section \ref{sec:2.2}, we introduce the basics of the Hamilton-Krein index theory, which culminates in an easy to apply Corollary \ref{yuri}, which allow us, in certain cases, to conclude spectral stability based on the Vakhitov-Kolokolov criteria. This is followed by a few useful lemmas, in particular the Sturm-Liouville theory in the fractional case, see Lemma \ref{hmj}.  In Section \ref{sec:3}, we present the variational construction, together with a selection of additional spectral properties for the operators $\cl_\pm$, as well as properties of the Lagrange multipliers $\om_\la$.  In Section \ref{sec:4} we show the non-degeneracy of the waves - the proof proceeds in two steps, first we establish in Lemma \ref{weaknond} the weak non-degeneracy, using properties of the functions $m, \om$. Next, we use the Sturm-Liouville theory available in this case to upgrade this to strong non-degeneracy - see Lemma \ref{sl} and the final stages of the proof immediately after. We finish this section by establishing spectral stability for the waves - note that while the orbital (nonlinear) stability results in the next section are stronger, they do require {\it a priori} well-posedness assumptions. Finally, in Section \ref{sec:5}, we show the orbital stability of the waves, both for fKdV and fNLS. Note that for that part, we employ a direct contradiction argument that  does not require continuity of the maps $\om\to \phi_\om$ or $\la\to \phi_\la$, as this is an open question as of this writing.

  \section{Preliminaries}
  \label{prelim}
  Introduce the Lebesgue and Sobolev spaces as usual, $\|f\|_{L^p[-T,T]}=\left(\int_{-T}^T |f(x)|^p dx\right)^{1/p}$, $1\leq p<\infty$. For the Fourier coefficients, taken $\hat{f}(k):=\f{1}{\sqrt{2 T}} \int_{-T}^T f(x) e^{- i \pi k  \f{x}{T}} dx$, one can define the $H^s $ norms via the standard
  $$
  \|f\|_{H^s} = \left(\sum_{k=-\infty}^ \infty (1+|k|^2)^s  |\hat{f}(k)|^2 \right)^{\f{1}{2}}. 
  $$
  Also, introduce $H^\infty=\cap_{k=1}^\infty H^k$. 
  Here is an interesting Sobolev embedding, which will be useful for us, see Lemma \ref{le:se} below, 
  \begin{equation}
  \label{sem} 
  \|f\|_{H^{-a}[-T,T]}\leq C_{a, T} \|f\|_{L^1[-T,T]},
  \end{equation}
  whenever $a>\f{1}{2}$. Indeed, we have 
  \begin{eqnarray*}
  \left(\sum_k \f{|\hat{f}(k)|^2}{<k>^{2 a}}\right)^{1/2}  \leq 
  C_a \sup_k |\hat{f}(k)| \le C_a \|f\|_{L^1[-T,T]}.
  \end{eqnarray*}
  \subsection{A posteriori  smoothness of weak solutions of fractional elliptic equations}
  In this section, we show that predictably, weak solutions to elliptic equations must be smoother  than initially required, as to be solutions in a stronger sense. We work with the underlying elliptic equation \eqref{30}, but one can easily extend the results below, by simply following our scheme.
  \begin{definition}
  	\label{defi:10} We say that $\phi\in L^2[-T,T]$ is a distributional  solution of \eqref{30}, if for every test function $h\in H^\infty[-T,T]$, one has the identity 
  	$$
  	\dpr{\phi}{\La^\al h}+ \om \dpr{\phi}{h}-\dpr{\phi^2}{h}+a\dpr{1}{\phi}=0.
  	$$
  	Note that $\dpr{\phi^2}{h}$ makes sense, since $\phi^2\in L^1$, while $h\in L^\infty$. 
  	
  \end{definition}
\noindent  We have the following {\it a posteriori} smoothness result.
  \begin{proposition}
  	\label{prop:apost}
  Let $\al>\f{1}{2}$. Then,  the distributional solutions $\phi$ of \eqref{30} belong to $H^\infty([-T,T])$.
  \end{proposition}
  \begin{proof}
  	We add  $A \phi$ to both sides of \eqref{30}, where $A$ is a large positive constant, say $A=|\om|+1$. Thus, the equation becomes
  	$A \phi+ \phi^2 - a=(\La^\al+\om +A) \phi $. Note $\si(\La^\al+\om +A)=\{\left(\f{\pi |k|}{T}\right)^\al+\om+A, k=0, \pm 1, \ldots \}\subset [1, \infty)$, whence $\La^\al+\om +A$ is invertible on $L^2[-T,T]$. Also, its inverse clearly improves the  regularity of its input by $\al$ derivatives. In other words, $((\La^\al+\om +A)^{-1}: H^s\to H^{s+\al}$. 
  	
Introduce  $\tilde{\phi}:=(\La^\al+\om +A)^{-1}[A \phi+ \phi^2 - a]$. This is of course nothing but the formal solution of \eqref{30}, that is $\phi$, but we are about to prove this rigorously. First, observe that since
$(\La^\al+\om +A)^{-1}: L^2\to H^\al$, we have that $\tilde{\phi}\in H^\al$. Then,  for every test function $h$, we have
$$
\dpr{\tilde{\phi}}{(\La^\al+\om +A)  h}=\dpr{A \phi+ \phi^2 - a}{h} = \dpr{\phi}{(\La^\al+\om +A)  h}
$$
It follows that $\tilde{\phi}=\phi$, in sense of distributions, since $(\La^\al+\om +A)(H^\infty)=H^\infty$. 
Thus, $\phi \in H^\al$. One can now bootstrap this to $H^\infty$, since once we know $\phi \in H^\al$, then $A \phi+ \phi^2 - a\in H^{\al}$, because of Sobolev embedding. But then $(\La^\al+\om +A)^{-1}: H^\al \to H^{2\al}$, so $\phi=\tilde{\phi}\in H^{2\al}$ and so on.
  \end{proof}
  Note that the variational solutions that we  produce will be distributional  solutions of \eqref{30}.   Thus, such solutions will be in the class $H^\infty$, as a consequence of the {\it a posteriori} smoothness results in Proposition \ref{prop:apost}.
  \subsection{Some basic results of the instability index theory}
  \label{sec:2.2}
  In this section, we present some results about the solvability of eigenvalue problems of the form \eqref{111} and \eqref{121}. In fact, there is a more general theory developed for more general eigenvalue problems of this type,   we shall generally follow the presentation in \cite{LZ} as it suits our purposes the best. Namely, considering an eigenvalue problem of the form
  \begin{equation}
  \label{45}
  \cj \cl f = \la f,
  \end{equation}
  under the assumptions, that there exists a real-valued Hilbert space $\cx$ (with dot product $\dprr{\cdot}{\cdot}$ and an action between $\cx$ and $\cx^*$ given by $\dpr{\cdot}{\cdot}$), so that $\cj, \cl$ are (generally) unbounded operators as follows
  \begin{itemize}
  	\item $D(\cj)\subset X^*$ and $\cj: D(\cj) \to \cx $, so that $\cj^*=-\cj$,  in the sense that for every $u,v\in D(\cj)\subset X^*$,
  	$\dpr{\cj u^*}{v^*}=-\dpr{u^*}{\cj v^*}$.
  	\item $\cl: \cx\to \cx^*$ is a bounded and symmetric operator, in the sense that for every $u,v\in \cx$, $(u,v)\to \dpr{\cl u}{v}$ is bounded and symmetric form on $\cx\times \cx$.
  	\item $Ker[\cl]$ is finite dimensional and   the following  is a $\cl$ invariant decomposition
  	$$
  	\cx=\cx_-\oplus Ker[\cl]\oplus \cx_+, n(\cl):=dim(\cx_-)<\infty,
  	$$
 where for some $\de>0$, $\cl|_{\cx_-}\leq -\de$,  $\cl|_{\cx_+}\geq \de$. That is,
 $\dpr{\cl u_-}{u_-}\leq -\de \|u_-\|^2, \dpr{\cl u_+}{u_+}\geq \de \|u_+\|^2$, for every $u_\pm \in \cx_\pm$.
 \item
 $$
 \{f\in \cx^*: \dpr{f}{u}=0, \forall u \in \cx_-\oplus   \cx_+\}\subset D(\cj).
 $$
  \end{itemize}
  Then, for the (finite dimensional) generalized kernel $gKer[\cl]:=\{u\in \cx: (\cj\cl )^k u=0, k=1,2, \ldots\}\subset Ker[\cl]$, take the complement $\cm$, that is $gKer[\cl] = Ker[\cl]\oplus \cm$. Introduce the non-negative integer
  $$
  k_0^{\leq 0}(\cl):=\max\{dim(Z): Z\  \textup{subspace of}\  \cm : \dpr{\cl z}{z}< 0, \forall z\in Z\}.
  $$
   Theorem 2.3, \cite{LZ}) asserts that,  $k_{unstable}$ - the number of real unstable eigenvalues\footnote{counted with multiplicities}  for \eqref{45}, $k_c$ - the number of unstable eigenvalues in the first quadrant,   $k_i^{\leq 0}$ is the number of purely imaginary eigenvalues $\la= i \mu, \mu>0$,  with negative Krein signature,
   \begin{equation}
   \label{a:40}
   k_{unstable}+ 2 k_c+2k_i^{\leq 0}= n(\cl)- k_0^{\leq 0}(\cl),
   \end{equation}
   see (2.9), \cite{LZ} for precise definitions. In particular, if $n(\cl)=1$ and $k_0^{\leq 0}(\cl)\geq 1$, we will be able to conclude from \eqref{a:40} that all the terms on the left are zero, hence spectral stability.

   Next, we discuss the particular setup in the cases \eqref{11} and \eqref{12} respectively. For the fKdV, that is for the spectral problem \eqref{11}, we take $\cj=\p_x, D(\cj)=H^1[-T,T]$, $\cl=\cl_+, D(\cl)=H^\al[-T,T]$. The Hilbert space $\cx:= H^{\f{\al}{2}}[-T,T]$, so that we have the required bounds $\dpr{\cl u}{v}\leq C \|u\|_\cx \|v\|_\cx$. Clearly, the other conditions will be satisfied, once we check that $Ker[\cl]$ and $\cx_-$ are finite dimensional subspaces and $Ker[\cl]\subset H^1$. 
   
   For the fNLS spectral problem \eqref{12}, we take $\cj=\left(\begin{array}{cc}
   0 & 1  \\
   -1 & 0
   \end{array}\right)$, while $\cl= \left(\begin{array}{cc}
   \cl_+ & 0  \\
   0 & \cl_-
   \end{array}\right), D(\cl)=H^\al\times H^\al$, while $\cx=H^{\f{\al}{2}}\times H^{\f{\al}{2}}$.
   
   An easy corollary of this theory is 
   \begin{corollary}
   	\label{yuri} 
   	Assume that $n(\cl_+)=1$, while $\cl_-\geq 0$. If in addition, 
   	\begin{itemize}
   		\item Weak non-degeneracy holds, i.e. $\phi\perp Ker[\cl_+]$. In particular, $\cl_+^{-1} \phi$ is well-defined. 
   		\item The Vakhitov-Kolokolov index is negative: $\dpr{\cl_+^{-1} \phi}{\phi}<0$
   	\end{itemize}
   	then the eigenvalue problems \eqref{111} and \eqref{121} are spectrally stable.     
   \end{corollary}

  \subsection{A few  useful lemmas}
  In this section, we present some  lemmas, which will be used in the sequel. They are unrelated, so we put them in the order in which they are referred to  in the text. 
  
  The generalized Polya-Szeg\"o inequality is standard for the functions on $\rone$, and it states that among all functions, the decreasingly rearranged ones have the smallest $H^{\be}$ norms, as long as $\be\in (0,1]$. 
  We need such result for periodic functions,  one can find it for example in \cite{CJ}, Lemma A.1. 
  \begin{lemma}
  	\label{Szego}[Generalized Polya-Szeg\"o inequality]
  	For any $\be\in (0,1]$,
  \begin{equation}
  \label{70}
  	\int_{-1}^1 |\La^\be u(x)|^2 dx\geq \int_{-1}^1 |\La^\be u^*(x)|^2 dx.
  \end{equation}
  	That is, whenever $u\in H^\be[-1,1]$, then $u^*\in H^\be[-1,1]$ and in addition, \eqref{70} holds. Equality is achieved only when $u$ is bell-shaped, i.e. $u=u^*$.
  \end{lemma}
 \noindent  The following lemma was proved  in \cite{S1}.
  \begin{lemma}
  	\label{7}
  	Let $f:[a,b]\to \rone$ be a continuous function, that satisfies
  	$$
  	\limsup_{\eps \to 0+} \sup_{\la\in (a,b)} \f{f(\la+\eps)+f(\la-\eps)-2 f(\la)}{\eps^2}\leq 0.
  	$$
  	Then, $f$ is concave down.
  \end{lemma}
  The next  result is a variant of the well-known Sturm-Liouville oscillation theorem, but this time for fractional Schr\"odinger operator. It  was first obtained for operators acting on the line $\rone$, \cite{FL}. It was then extended for the periodic case, following similar ideas in \cite{HJ} (for the lowest  three eigenfunctions), 
  and then in   \cite{HMJ} for all eigenfunctions. 
  \begin{lemma}
  	\label{hmj} 
  	Let $V:[-T,T]\to \rone$ be a continuous function and $\al\in (0,2)$. Consider  the self-adjoint fractional Schr\"odinger operator   	$  	\ch= \La^{\al} + V$ with domain $D(\ch)=H^\al[-L,L]$. 
  	Let its spectrum\footnote{which consists entirely of eigenvalues with finite multiplicity} be ordered as follows 
  	$$
  	\la_0(\ch)<\la_1(\ch)\leq \la_2(\ch) \leq \ldots 
  	$$ 
  	Then, the corresponding eigenfunctions $\psi_n: \ch \psi_n = \la_n \psi_n$ have no more than $2 n$ changes of sign in the interval $[-T,T]$.  
  \end{lemma}
  The next lemma is about the mapping properties of Schr\"odinger operators $\ch$ of the type described in Lemma \ref{hmj} and its inverses, whenever they exist. First, for every $\la\in \rone$, $\la \notin \si(\ch)$, we have that $\ch-\la: H^\al \to L^2$, whence $(\ch-\la)^{-1}: L^2 \to D(\ch)=H^\al$. By taking adjoints, we also have $(\ch-\la)^{-1}: H^{-\al}\to L^2$, for $\la\in \rone\cap \rho(\ch)$. Taking into account the embedding $L^1[-T,T]\hookrightarrow H^{-\al}$, i.e. \eqref{sem},  we have shown 
  \begin{lemma}
  	\label{le:se} For $\al>\f{1}{2}$,  and $a\notin \si(\ch)$, we have $(\ch-\la)^{-1}: L^1[-T,T]\to L^2[-T,T]$. 
  	In addition, supposing  that for invariant subspace, $S\subset L^2[-T,T]$ of $\ch$, we have that $\la\notin \si_S(\ch)$.   That is, $(\ch-\la)^{-1}: S\to S$.  Then, 
  	$$
  	\|(\ch-\la)^{-1} f\|_{L^2\cap S}\leq C \|f\|_{L^1\cap S}
  	$$
  \end{lemma}

  \section{The variational construction}
  \label{sec:3}

  The classical way to produce solitary waves is to minimize energy, with respect to fixed $L^2$ norm. The result of this are the so-called normalized waves. In order to simplify the exposition, we shall work with $T=1$. Later on, we easily reduce to this case by a simple rescaling argument.   
  \begin{proposition}
  	\label{prop:10}
  	Let $\al \in \left(\f{1}{2}, 2\right]$ and $\la>0, a\in\rone$. Then, the  minimization problem
\begin{equation}
\label{40}
\left\{
\begin{array}{l}
\ce_a[\vp]:=\f{1}{2} \int_{-1}^1 |\La^{\al/2}\vp(x) |^2 dx  - \f{1}{3} \int_{-1}^1 |\vp(x)|^3 dx+ a\int_{-1}^1
|\vp(x)| dx \\
\int_{-1}^1 \vp^2(x) dx = \la
\end{array}
\right.
\end{equation}
has a bell-shaped solution, $\vp=\vp_{a,\la}$. Moreover, $\vp_{a, \la}$    satisfies, in a distributional sense, the Euler-Lagrange equation 
\begin{equation}
\label{456}
\La^\al \vp+  \om \vp - \vp^2 + a =0, -1\leq x\leq 1,
\end{equation}
for $\om=\om(\la, a; \vp)$, given by the either of the two formulas
\begin{eqnarray}
\label{50}
\om_{\la, a} &=&  \f{ \int_{-1}^1 \vp^3(x) dx - \int_{-1}^1 |\La^{\al/2}\vp(x) |^2 dx-
	a \int_{-1}^1 \vp(x) dx}{\la}, \\
\label{51}
 \om_{\la, a} &=&  \f{\la-2 a}{\int_{-1}^1 \vp(x) dx}
\end{eqnarray}
In addition, we have the following preliminary properties of the linearized operators
\begin{itemize}
	\item $\cl_+$ has exactly one negative eigenvalue, denoted by $-\si_\la^2$, which is simple, with a corresponding eigenfunction $\chi_\la$.
	In addition, $\cl_+|_{\{\vp_\la\}^\perp}\geq 0$.
	\item   For $a=0$, the operator  $\cl_-:=\La^\al+\om - \vp\geq 0$ satisfies $Ker(\cl_-)=span[\vp]$ and for some $\de>0$,
	$\cl_-|_{\{\vp\}^\perp}\geq \de Id$.
	\item For $a<0$, there exists $\de>0$, so that $\cl_-\geq \de Id$, while for $a>0$, $\cl_-$ has $n(\cl_-)=1$, with $\cl_-|_{\{\vp\}^\perp}\geq \de Id$.  In particular, for $a\neq 0$, $0\notin \si(\cl_-)$, in other words $\cl_-^{-1}$ exists.
\end{itemize}
\end{proposition}
{\bf Remarks:}  In the statement above, it is implicit that the Lagrange multiplier $\om$ may in fact depend on the particular minimizer $\vp$ as well. This is also related to the uniqueness issue for the solutions of the constrained minimization problem \eqref{40}.  More precisely, for given values of $\la>0, a\in \rone$, it is possible that there exist two solutions $\vp_{\la,a}, \tilde{\vp}_{\la,a}$ of \eqref{40}. Each of them will certainly satisfy the  Euler-Lagrange equation \eqref{456}, but may be two different Lagrange multipliers $\om, \tilde{\om}$. We cannot rule out neither of these possibilities in this article.

Our next result prepares some background information, needed later on in the arguments,  for the following function 
 $$
 m(\la):=\inf\limits_{\int_{-1}^1 \vp^2(x) dx = \la} \ce_a [\vp].
 $$
 Note that it is not {\it a priori} clear why $m$ is even finite for all $\la>0$, but this is established below. Also,   $m$ also depends on $a$, but we prefer not to emphasize this dependence. 
 \begin{proposition}
 	\label{prop:11} The function $m$ has the following properties
 	\begin{itemize}
 		\item $m$ is finite everywhere, that is $m(\la)>-\infty$ for every $\la>0$,
 		\item $m$  is a locally Lipschitz, and its derivative, which exists at least a.e.,  can be computed  to be $m'(\la)=-\f{\om_\la}{2}$,
 		\item $m$ is concave down. 
 	\end{itemize}
 	In particular, at all points in the full measure subset  $\ca:=\{\la\in\rone_+: m'(\la) - \textup{exists}\}$,  the function $\om=\om(\la)$ is independent on the concrete minimizer $\vp$, as a derivative of $-2m(\la)$. 
 	 
 	Regarding the function $\la\to \om_\la$, 
 		\begin{itemize}
 			\item For $a\leq 0$, $\om_\la>0$ for all $\la$,
 			\item For $a>0$,   $\om_\la<0, \la \in (0,2a)$ and
 			$\om_\la>0, \la \in (2a, \infty)$.
 			\item  $\la\to \om_\la,  \la\in \ca$ is non-decreasing.  In fact, its first derivative\footnote{which is guaranteed to exists at least a.e. in view of the monotonicity} satisfies 
 			$$
 			\om'(\la)>\f{\si_\la^2}{2\dpr{\chi_\la}{\vp_\la}}.
 			$$
 			That is $\om'(\la)>0$ for a.e. $\la>0$.  
 			\item Even outside of $\ca$, the function $\la\to \om(\la, \vp_\la)$ is non-decreasing. More precisely, suppose $0<\la_1<\la_2$, with corresponding minimizers $\vp_{\la_1}, \vp_{\la_2}$. Then, 
 			$$
 			\om(\la_1, \vp_{\la_1})\leq 	\om(\la_2, \vp_{\la_2}). 
 			$$ 
 		\end{itemize}
 \end{proposition}

  \noindent We prove these results  over the course of the Section \ref{sec:3}. 
  	
  	\subsection{Well-posedness and existence of minimizers  for  the variational problem \eqref{40}}
  	\label{sec:3.1}
  Let $\eps>0$ be an arbitrary real. Then,  by  the Sobolev embedding, for any  $\vp\in H^{\al/2}$ satisfying the constraint,
  	\begin{equation}
  	\label{60}
  		\|\vp\|_{L^3}^3\leq C \|\vp\|_{H^{\f{1}{6}}}^3 \leq C \|\vp\|_{H^{\al/2}}^{\f{1}{\al}} \|\vp\|_{L^2}^{3-\f{1}{\al}}=
  		C \la^{\f{3}{2}-\f{1}{2 \al}}\|\vp\|_{H^{\al/2}}^{\f{1}{\al}}\leq \eps \|\La^{\al/2} \vp\|_{L^2}^2+C_{\eps, \la},
  	\end{equation}
  where in the last step, we have used the Young's inequality. Clearly, $\inf$ in the constrained minimization problem \eqref{50} is bounded from below, hence the problem is well-posed.

  Pick a minimizing sequence, that is $\vp_n \in H^{\al/2}$, so that $\ce[\vp_n]\to m(\la)$.
  We will show that the sequence is compact in $L^2$ and subsequently in all $L^p, p\in (2, \infty)$. We have that for all large enough $n$,
  $\ce[\vp_n]<m(\la)+1$. It follows that
  $$
  \|\La^{\al/2} \vp_n\|_{L^2}^2\leq 2 \ce[\vp_n]+\f{2}{3} \int_{-1}^1 \vp_n^3(x) dx - 2 a \int_{-1}^1 |\vp_n(x)| dx.
  $$
  By \eqref{60} and Cauchy-Schwartz, the right hand side can be estimated as follows
  \begin{equation}
  \label{62}
   \|\La^{\al/2} \vp_n\|_{L^2}^2\leq 2(m(\la)+1) + \f{2}{3}(\|\La^{\al/2} \vp_n\|_{L^2}^2+C_{\eps, \la})+ 4 |a| \sqrt{\la}.
  \end{equation}
  Let us reiterate that this estimate holds whenever $\|\vp\|^2=\la$.  Hiding $\|\La^{\al/2} \vp_n\|_{L^2}^2$ behind the left-hand side,  leads to an {\it a priori} estimate on $\|\La^{\al/2} \vp_n\|\leq C_\la$.

  Since $H^{\al/2}[-1,1]$ compactly embeds into $L^2[-1,1]$, so $\{\vp_n\}$ is a compact in $L^2[-1,1]$. Taking a convergent subsequence $\{\vp_{n_k}\}_{k=1}^\infty$, we find that its limit $\vp:=\lim_k \vp_{n_k}$ satisfies $\|\vp\|_{L^2}^2=\la$. By Gagliardo-Nirenberg's inequality, the same sequence is compact in any $L^p, p>2$ spaces. By H\"older's and Gagliardo-Nirenberg, it is also compact in any $L^q, q\in [1,4)$ ( recall $\sup_n \|\vp_n\|_{H^\f{\al}{2}}<\infty$, $H^{\f{\al}{2}} \hookrightarrow L^4$).  In particular,
  $$
   \int_{-1}^1 \vp_{n_k}(x) dx \to \int_{-1}^1 \vp(x) dx,\ \  \int_{-1}^1 |\vp_{n_k}(x)|^3 dx\to
  \int_{-1}^1 |\vp(x)|^3 dx.
  $$
  Finally, by the weak convergence $H^{\al/2}$, $\vp_{n_k}\to \vp$, we have by the lower semi-continuity of the norms with respect with weak convergence 
  $$
  \liminf_k \int_{-1}^1 |\La^{\al/2}\vp_{n_k}(x) |^2 dx\geq \int_{-1}^1 |\La^{\al/2}\vp(x) |^2 dx.
  $$
  But then,
  $$
 m(\la)=\lim_k\ce[\vp_{n_k}]\geq   \liminf_k \ce[\vp_{n_k}]\geq \ce[\vp],
  $$
  which is a contradiction (recall $\int\vp^2(x) dx = \la$), unless $\ce[\vp]=m(\la)$. In addition, since it must be that $\lim_k \int_{-1}^1 |\La^{\al/2}\vp_{n_k}(x) |^2 dx=\int_{-1}^1 |\La^{\al/2}\vp(x) |^2 dx$, it follows that
  $\lim_k \|\vp_{n_k}-\vp\|_{H^{\f{\al}{2}}}=0$.  Thus, $\vp$ is a minimizer.
  We observe that the  minimizer is necessarily bell-shaped, by the generalized Polya-Szeg\"o's inequality, \eqref{70}.

  Note that we have shown in particular, that each minimizing sequence has an $H^{\f{\al}{2}}$  convergent subsequence, which converges to a minimizer.  

 \subsection{Euler-Lagrange equation}  We now derive the Euler-Lagrange equation \eqref{456}. Let $\eps>0$ and take any test function
  $h\in H^\infty[-1,1]$. Consider
  $$
  g(\eps):=\ce\left[\sqrt{\la} \f{\vp+\eps h}{\|\vp+\eps h\|}\right]\geq g(0)=\ce[\vp].
  $$
  Observe that
  \begin{eqnarray*}
\|\vp+\eps h\|^q &=& \la^{q/2}+\eps q \la^{q/2-1} \dpr{\vp}{h} +O(\eps^2),
  \end{eqnarray*}
  whence
\begin{eqnarray*}
	\f{\la}{2\|\vp+\eps h\|^2} \int_{-1}^1 |\La^{\al/2} (\vp+\eps h)|^2dx &=&  \f{1}{2}
	\int_{-1}^1 |\La^{\al/2}\vp|^2dx+\eps[\dpr{\La^\al \vp -\f{\|\La^{\al/2} \vp\|^2}{\la} \vp}{h}]+O(\eps^2)  ]   \\
  \f{\la^{3/2}}{3\|\vp+\eps h\|^3} \int_{-1}^1 (\vp+\eps h)^3 dx &=&  \f{1}{3} \int_{-1}^1 \vp^3 (x) dx+\eps[\dpr{\vp^2}{h} - \f{\dpr{\vp}{h}}{\la} \int \vp^3(x) dx]+O(\eps^2), \\
  \f{a\sqrt{\la}}{\|\vp+\eps h\|} \int_{-1}^1 (\vp(x)+\eps h(x)) dx &=& a \int_{-1}^1 \vp (x) dx+
  \eps[a\dpr{1}{h}-\f{a}{\la} \dpr{\vp}{h} ]+O(\eps^2),
\end{eqnarray*}
  Putting everything together, we obtain
  $$
  g(\eps)=g(0)+\eps[\dpr{\La^{\al}\vp - \vp^2+a-\f{\|\La^{\al/2} \vp\|^2-\int \vp^3+a \int \vp}{\la}\vp}{h}+O(\eps^2).
  $$
  It follows that \eqref{456} is satisfied, in a weak sense, with $\om$ given by \eqref{50}.
  We now turn to the statements regarding the linearized operators $\cl_\pm$. 
  Introduce  the following notation - for a self-adjoint operator $\cm$, which is bounded from below and which has at most finitely many negative eigenvalues, denote by $n(\cm)$ the number of the negative eigenvalues, counted with multiplicities. 
  \subsection{Spectral properties of $\cl_\pm$}
  \label{sec:3.3}

  We start with the spectral properties of $\cl_+$.  We use again the property that $g$ attains its minimum at $\eps=0$. In order to simplify the
  argument, take the test function $h$, so that $h\perp \vp, \|h\|_{L^2}=1$. Note that this implies
  $\|\vp+\eps h\|_{L^2}^2=\la+\eps^2$, whence
  $$
  \|\vp+\eps h\|_{L^2}^q=\la^{\f{q}{2}}+\f{q}{2} \la^{\f{q}{2}-1} \eps^2+o(\eps^2).
  $$
  The expansion of $g(\eps)$ around zero takes the form
  \begin{eqnarray*}
g(\eps)&=& \ce\left[\sqrt{\la} \f{\vp+\eps h}{\|\vp+\eps h\|}\right]=g(0)+  \eps\dpr{\La^{\al}\vp - \vp^2+a}{h}+\f{1}{2}[\|\La^{\al/2} \vp\|^2+\eps^2 \|\La^{\al/2} h\|^2)][1-\f{1}{2\la} \eps^2] \\
&-& \f{1}{3}[\int \vp^3(x)+3\eps^2 \dpr{\vp h}{h}][1-\f{3}{2\la} \eps^2]dx +a (\int \vp(x) dx) (1-\f{1}{2\la} \eps^2)+o(\eps^2).
  \end{eqnarray*}
  Clearly, $\dpr{\La^{\al}\vp - \vp^2+a}{h}=\dpr{\La^{\al}\vp - \vp^2+a+\om \vp}{h}=0$, by the Euler-Lagrange equation. Thus, we can rewrite the last identity as
  $$
  g(\eps)-g(0)=\f{\eps^2}{2}\left(\dpr{\La^\al h}{h} - 2\dpr{\vp h}{h}+\om \right)+o(\eps^2).
  $$
  Recalling that $\|h\|=1$, $\left(\dpr{\La^\al h}{h} - 2\dpr{\vp h}{h}+\om \right)=\dpr{\cl_+ h}{h}$. Since $0$ is a local minimum for the function $g$, we conclude that $\dpr{\cl_+ h}{h}\geq 0$. Thus,
  $$
  \cl_+|_{\{\vp\}^\perp}\geq 0,
  $$
  whence we deduce that $\cl_+$ has at most one negative eigenvalue, or $n(\cl_+)\leq 1$. On the other hand, by differentiating the Euler-Lagrange equation in $x$, we obtain
    $\cl_+[\vp']=0$, hence zero is an eigenvalue. Note however that $\vp'$ changes sign in $[-1,1]$, hence it is not the eigenfunction corresponding to the smallest eigenvalue. It follows that there is a negative eigenvalue or $n(\cl_+)=1$.

    The claims about $\cl_-$ follow easily in the case $a=0$.   By direct evaluation,  $\cl_-[\vp]=0$ (this is simply \eqref{456}), so $0$ is an eigenvalue. Since   $\cl_-|_{\{\vp\}^\perp}>\cl_+|_{\{\vp\}^\perp}\geq 0$, we conclude that $0$ is at the bottom of the spectrum. 

    In the case $a\neq 0$, we observe that $\cl_- \vp= -a$, this is again an instance of \eqref{456}.  Let now $a<0$. Assuming that the smallest eigenvalue is $-\si^2, \si\geq 0$, take $\Psi$ to be its (necessarily positive, according to Sturm-Liouville's theory) eigenfunction, $\cl_-\Psi=-\si^2 \Psi$. Take a dot product of this last identity with $\vp$. We have
    $$
   0< -a\dpr{\Psi}{1}=\dpr{\cl_-\Psi}{\vp}=-\si^2 \dpr{\Psi}{\vp}\leq 0,
    $$
    all due to the $\Psi>0, \vp>0$, $a<0$. So, a contradiction is reached, which implies $\cl_->0$.

    In the case $a>0$, we observe that $\dpr{\cl_- \vp}{\vp}=-a\dpr{1}{\vp}<0$, whence $\cl_-$ has at least one negative eigenvalue. Since $\cl_->\cl_+$ and $n(\cl_+)=1$, it follows that $\cl_-$ has exactly one negative eigenvalue and moreover, $\cl_-|{\{\vp\}^\perp}>\cl_+|{\{\vp\}^\perp}\geq 0$, so $0\notin \si(\cl_-)$.

    \subsection{Properties of $m(\la), \om(\la)$}
    \label{sec:3.4} 
  Recall  that we have shown that \eqref{40} is well-posed and solvable. We have also established a number of useful spectral properties of $\cl_\pm$. We now turn to the proof of Proposition \ref{prop:11}.

We start with the observation, that with  the test function  $u=\sqrt{\f{\la}{2}}$, we arrive for the following (very rough) estimate for $m(\la)$, namely $m(\la) \leq -\f{\la^{3/2}}{3\sqrt{2}} +a\sqrt{2\la}$.  In addition, we have derived various {\it a priori} estimates on the minimizers in the form
  $\|\La^{\al/2} \vp_\la\|_{L^2}\leq C_\la$, see for example \eqref{62}. If we enter the just obtained estimate for $m(\la)$, we arrive at a explicitly computable and continuous in $\la$  bound $C_\la$. In view of all this, we can setup the variational problem in the form
  $$
  m(\la):=\inf\limits_{\int_{-1}^1 u^2(x) dx = \la: \|\La^{\al/2} u\|_{L^2}\leq 2 C_\la} \ce[u].
  $$
  Introducing the new variable $U: u=\sqrt{\la} U$, consider a new function
  $$
  k(\la):=\f{m(\la)}{\la}= \inf\limits_{\int_{-1}^1 U^2(x) dx = 1: \|\La^{\al/2} U\|_{L^2}\leq  D_\la}
   \f{1}{2} \int_{-1}^1 |\La^{\al/2} U(x) |^2   - \f{\sqrt{\la}}{3} \int_{-1}^1 |U(x)|^3 + \f{a}{\sqrt{\la}} \int_{-1}^1
   |U(x)|.
  $$
  where $D_\la:=2 \la^{-1/2} C_\la$ is  also continuous.

  We will now show that $\la\to k(\la)$ is locally Lipschitz, whence $m(\la)$ will be locally Lipschitz as well. Considering the functional over which we need to minimize for the construction of $k(\la+\de)$, for small $\de$, we have for every $U$ in the constrained set
  \begin{eqnarray*}
 & & \f{1}{2} \int_{-1}^1 |\La^{\al/2} U(x) |^2 dx  - \f{\sqrt{\la+\de}}{3} \int_{-1}^1 |U(x)|^3 dx+ a(\la+\de)^{-1/2} \int_{-1}^1  |U(x)| dx  = \\
 &=& \f{1}{2} \int_{-1}^1 |\La^{\al/2} U(x) |^2 dx  - \f{\sqrt{\la}}{3} \int_{-1}^1 |U(x)|^3 dx+ a\la^{-1/2} \int_{-1}^1
 |U(x)| dx + E_{\de, \la},
  \end{eqnarray*}
  where
  $$
  |E_{\de, \la}|\leq C|\de| (\la^{-1/2}+ \la^{-3/2})(\|U\|_{L^3}^3+ \|U\|_{L^1})\leq C |\de| (1+D_{\la+\de}^3),
  $$
 since we have assume that $U$ is in the constrained set for $k(\la+\de)$ and hence by H\"older's and Sobolev embedding  $\|U\|_{L^3}+ \|U\|_{L^1}\leq C \|\La^{\al/2} U\|\leq C D_{\la+\de}$. Taking $\inf\limits_{\int_{-1}^1 U^2(x) dx = 1: \|\La^{\al/2} U\|_{L^2}\leq  D_\la} $, we obtain
 $$
 k(\la)-C |\de| (1+D_{\la+\de}^3)\leq  k(\la+\de)\leq k(\la)+C |\de| (1+D_{\la+\de}^3),
 $$
  This implies Lipschitzness of the mapping $\la\to k(\la)$, once we take into account that
   $\la\to D_\la$ is continuous and hence  locally bounded. Thus, $\la\to m(\la)$ is locally Lipschitz and it has a derivative almost everywhere. In fact, we can compute its derivative, whenever it exists,  explicitly. 
   \begin{lemma}
   	\label{le:65}
   The function $m$ is differentiable a.e. in $\rone_+$ and there is the formula 
   		$$
   		m'(\la)=-\f{\om_\la}{2}.
   		$$
   		In particular,  since $m$ is absolutely continuous, it can be recovered from its a.e. derivative.  Namely for every $0<\la_1<\la_2$, there is  
   \begin{equation}
   \label{65} 
   	m(\la_2)-m(\la_1)=-\f{1}{2} \int_{\la_1}^{\la_2} \om(\la) d\la.
   \end{equation}
 Finally,  $m$ is concave down. In particular, $m$ is twice differentiable a.e. in $\la$ and $m''(\la)\leq 0$. Moreover, for every $0<\la_1<\la_2<\infty$,  with  corresponding minimizers $\vp_{\la_1}, \vp_{\la_2}$ 
 \begin{equation}
 \label{517} 
 \om(\la_1, \vp_{\la_1})\leq   \om(\la_2, \vp_{\la_2}). 
 \end{equation}
   \end{lemma}
   {\bf Remark:} Note that the concavity of $m$ implies that the function $\la\to \om_\la=-2m'(\la)$, (which is defined a.e.) is non-decreasing. The property \eqref{517} is an extension of this, as it claims that even when $\la\to \om$ may depend on the particular minimizer $\vp_\la$, it is still a non-decreasing function of $\la$.  
   \begin{proof}
   Starting with a minimizer $\vp_\la$, we have by definition that for all $\eps\in \rone$ and test functions $h$,
   $$
   \ce(\vp_\la+\eps h)\geq m(\|\vp_\la+\eps h\|^2).
   $$
   But expanding in powers of $\eps$, we see that
\begin{eqnarray*}
   \ce(\vp_\la+\eps h) &=& \ce(\vp_\la)- \eps \om_\la \dpr{\vp_\la}{h}+\f{\eps^2}{2}\dpr{(\cl_+-\om_\la)h}{h}+O(\eps^3), \\
   m(\|\vp_\la+\eps h\|^2) &=& m(\la+2\eps \dpr{\vp_\la}{h}+\eps^2 \|h\|^2).
\end{eqnarray*}
Taking into account $\ce(\vp_\la)=m(\la)$, we arrive at
\begin{equation}
\label{500}
m(\la+2\eps \dpr{\vp_\la}{h}+\eps^2 \|h\|^2)\leq m(\la)- \eps \om_\la \dpr{\vp_\la}{h}+\f{\eps^2}{2}\dpr{(\cl_+-\om_\la)h}{h}+O(\eps^3)
\end{equation}
   Ignoring for a second all terms in the form $O(\eps^2)$, we can see that whenever $m'(\la)$
   exists\footnote{which is at least a.e. at this point, since it was established that $m$ is Lipschitz}, we can compute it as follows fix $h=\vp_\la$,  for $\eps>0$, divide \eqref{500} by $2\la \eps+\la \eps^2>0$ for $0<\eps<<1$, so
\begin{equation}
\label{502}
   \f{m(\la+2\la \eps  +\la \eps^2) -m(\la)}{2\la \eps+\la \eps^2} \leq -\f{\eps \la  \om_\la }{2 \la \eps+\la \eps^2}
   	 +O(\eps),
\end{equation}
  It follows that  $m'(\la)\leq -\f{\om_\la}{2}$. Similarly, for $\eps<0$, we  divide   by $2\la \eps+\la \eps^2<0$ for $\eps<0, |\eps|<<1$, so that after taking limit $\lim_{\eps\to 0-}$, we get the opposite inequality
  $m'(\la)\geq -\f{\om_\la}{2}$. Altogether, $m'(\la)=-\f{\om_\la}{2}$. 

  Next, we show that $m$ is concave down. To this end,  apply 
   \eqref{500}  for $h=\f{\chi_\la}{2 \dpr{\chi_\la}{\vp_\la}}$, where $\chi_\la: \|\chi_\la\|=1$ is the eigenfunction, corresponding to the negative eigenvalue of $\cl_+$, that is $\cl_+ \chi_\la=-\si_\la^2 \chi_\la$
   $$
   m(\la+\eps+ \eps^2 \|h\|^2) - m(\la)\leq -\f{\eps}{2} \om_\la  -   \f{\eps^2}{2} \om_\la \|h\|^2- 
   \f{\eps^2}{2} \si_\la^2 \|h\|^2+O(\eps^3)
   $$
   Introduce now $\de:=\eps+ \eps^2 \|h\|^2$, so that the previous inequality reads 
   \begin{equation}
   \label{508} 
   m(\la+\de)-m(\la)\leq -\f{\eps_\de}{2} \om_\la  -   \f{\eps_\de^2}{2} \om_\la \|h\|^2- 
   \f{\eps_\de^2}{2} \si_\la^2 \|h\|^2+O(\de^3),
   \end{equation}
   where $\eps_\de$ is given by the quadratic equation formula 
\begin{equation}
\label{509} 
 \eps_\de=\f{-1+\sqrt{1+4\de \|h\|^2}}{2\|h\|^2}= \de- \de^2 \|h\|^2+O(\de^3). 
\end{equation}
   Applying \eqref{508} to $-\de$ instead of $\de$ and adding the result to \eqref{508} yields 
   \begin{equation}
   \label{510} 
   m(\la+\de)+m(\la-\de) - m(\la)\leq -\f{\eps_\de+\eps_{-\de}}{2} \om_\la  -   \f{\eps_\de^2+\eps_{-\de}^2}{2} \om_\la \|h\|^2- 
  \f{\eps_\de^2+\eps_{-\de}^2}{2}  \si_\la^2 \|h\|^2+O(\de^3). 
   \end{equation}
   Taking into account the asymptotics \eqref{509}, we conclude 
   \begin{equation}
   \label{512} 
   m(\la+\de)+m(\la-\de) - m(\la)\leq  
   -\de^2 \si_\la^2 \|h\|^2+O(\de^3). 
   \end{equation}
   Dividing by $\de^2$, taking $\sup_{\la\in (a,b)}$ on any interval $(a,b)\subset \rone_+$ and taking a limit in $\de\to 0+$ allows to conclude 
   $$
   \lim_{\de\to 0+} \sup_{\la\in (a,b)} \f{m(\la+\de)+m(\la-\de) - m(\la)}{\de^2}\leq 0
   $$
   invoking Lemma \ref{7}, we derive that $m$ is concave down on $\rone_+$. This of course means that the $\om(\la)$ is non-decreasing, differentiable a.e. in $\la$  and from \eqref{512}, we can in fact derive the estimate a.e. in $\la$ 
   $$
  \om'(\la)=-2m''(\la) > \f{ \si_\la^2}{2\dpr{\chi_\la}{\vp_\la}^2}>0
   $$
  Now that we know that $m$ is concave down, it means that it has a left and right derivatives everywhere.  Note that even when $m$ does not have a derivative, we can still take limits in \eqref{502} (and its analog for $\eps<0$) to obtain 
  \begin{equation}
  \label{513} 
   m'(\la+)\leq -\f{\om(\la, \vp_\la)}{2}\leq m'(\la-).
 \end{equation}
    In particular, for every $0<\la_1<\la_2<\infty$, we have from \eqref{513} 
   $$
   \om(\la_1, \vp_{\la_1})\leq -2m'(\la_1+)\leq -2m'(\la_2-) \leq \om(\la_2, \vp_{\la_2}). 
   $$
   Combining the last estimate with \eqref{513} provides a direct proof that $m'$ is non-increasing function as well. 
   \end{proof}

   \section{Non-degeneracy of the waves and spectral stability} 
   \label{sec:4} 
   Non-degeneracy of the waves not only plays an important role in the stability considerations, 
   but it is of interest in its own. In particular, it always seems to be  an important first step towards uniqueness of the waves, as solutions to the corresponding profile equations, e.g. \eqref{30}. We start with a less ambitious task, which turns out to be the main step towards the non-degeneracy, we call it {\it weak non-degeneracy}.  
  \subsection{Weak non-degeneracy of $\vp_\la$} 
   \begin{lemma}
   	\label{weaknond} The constrained minimizers $\vp_\la$ produced in Proposition \ref{prop:10} enjoy the weak \\ non-degeneracy property,  that is $\vp_\la\perp Ker[\cl_+]$. 
   \end{lemma}
  \begin{proof}
  	We first establish that $\dpr{\cl_+ \vp_\la}{\vp_\la}<0$. We have, using \eqref{50},  
  	$$
  	\dpr{\cl_+ \vp_\la}{\vp_\la}=\|\La^{\al/2} \vp\|^2 +\om \la - 2\int\vp^3= - \int\vp^3-a\int \vp. 
  	$$
  	This is clearly negative if $a\geq 0$, but  the sign of it is not easily determined, if $a<0$. 
  	
  	In order to show this, we shall need to see first that $m(0+)=0$. As before, trying the function $u=\sqrt{\f{\la}{2}}$ yields a bound from above,  $m(\la) \leq -\f{\la^{3/2}}{3\sqrt{2}} +a\sqrt{2\la}$, which implies $m(0+) \leq 0$. For the bound from below, we use \eqref{60}, which implies that for all $\eps>0$, 
  	$\|u\|_{L^3}^3\leq \eps \|\La^{\al/2} u\|^2+ C_\eps \la^{\f{3\al-1}{2\al-1}}$, which in turn implies 
  	$$
  	m(\la)\geq \inf_{\|u\|^2=\la}[\f{1}{4} \|\La^{\al/2} u\|^2 - a \int_{-1}^1 |u| dx]-C \la^{\f{3\al-1}{2\al-1}}
  	\geq -C(\sqrt{\la}+\la^{\f{3\al-1}{2\al-1}}),
  	$$
  	Taking $\lim_{\la\to 0+}$ yields the bound $m(0+) \geq 0$, and subsequently $m(0+)=0$.  Now, using \eqref{65}, with $\la_1=0+$,  and the fact that $\la\to \om_\la$ is non-decreasing (i.e. the property  \eqref{517}) 
  $$
  -2m(\la)=\int_0^\la \om(\mu) d\mu\leq \la \om(\la-)\leq  \la \om(\la, \vp_{\la}). 
  $$
  	It follows that 
  	$$
  	0\leq 2 m(\la)+\la \om(\la, \vp_{\la}) =(\|\La^{\al/2} \vp\|^2 -\f{2}{3}\int \vp^3+2 a\int \vp) + (\int \vp^3 -\|\La^{\al/2} \vp\|^2 -a\int \vp) = \f{1}{3} \int \vp^3+a\int \vp.
  	$$
  	In particular, 
  	$$
  		\dpr{\cl_+ \vp}{\vp}=- \int\vp^3-a\int \vp<-(\f{1}{3} \int \vp^3+a\int \vp)\leq 0.
  	$$
  	We now apply the property $\cl_+|_{\{\vp\}^\perp}\geq 0$. More concretely, take $h\in Ker[\cl_+]$ and consider $\tilde{h}:=h-\la^{-1} \dpr{h}{\vp_\la}\vp_{\la}\perp \vp_\la$. It must be that 
  	$$
  	0\leq \dpr{\cl_+ \tilde{h}}{\tilde{h}}=\la^{-2} \dpr{h}{\vp_\la}^2 \dpr{\cl_+ \vp}{\vp} 
  	$$ 
   Assuming  $\dpr{h}{\vp_\la}\neq 0$, this leads to a contradiction, as the right-hand side is strictly negative. Thus, $\dpr{h}{\vp_\la}=0$ and the weak non-degeneracy is established.

  \end{proof}
  \subsection{Non-degeneracy of $\vp_\la$: conclusion of the proof} 
   We now continue with the goal of establishing that the wave $\vp_\la$ is non-degenerate, that is $Ker[\cl_+]=span[\vp_\la']$. 
  	Note that we always have $\vp_\la'\subset Ker[\cl_+]$. 	We claim that  $Ker[\cl_+]$ is at most two dimensional. Indeed, we know already that $n(\cl_+)=1$, so $\la_0(\cl_+)<0$. Since $0$ is an eigenvalue, it must be that $\la_1(\ch)=0$.  By bell-shapedness, one of the corresponding eigenfunctions, $\vp_\la'$  is an odd function, which has exactly one zero, at $x=0$. Since $\cl_+$ is a fractional Schr\"odinger operator with even potential, the linearly independent eigenfunctions may be taken to be either even or odd.  
  	
  	By the Sturm-Liouville's theory for the fractional periodic Schr\"odinger operators, see  Lemma \ref{hmj}, we have that the eigenfunctions corresponding to the zero eigenvalue have at most two zeroes in $[-T,T]$. Clearly, there cannot be another odd eigenfunction (other than $\vp_\la'$), since it would have to have exactly one zero, which happens at $x=0$, and as such, it cannot possibly be orthogonal to $\vp_\la'$. Thus, there could be another eigenfunction, say $\Psi_\la: \|\Psi_\la\|_{L^2}=1$, which is  even and which has exactly two zeros (since it cannot have one zero), at say $\pm b, b\in (0,T)$. Note that similar to Proposition \ref{prop:apost}, it can be shown that $\Phi_\la\in H^{\infty}[-T,T]$. 
  	Thus, we have proved the following preliminary result
  	\begin{lemma}
  		\label{sl} 
  		For the fractional Schr\"odinger operator $\cl_+$, we have that either $Ker[\cl_+]=span[\vp_\la']$ or $Ker[\cl_+]=span[\vp_\la', \Psi_\la]$, where  $\Psi_\la:[-T,T]\to \rone$ is a smooth even function, with exactly two zeroes,  $\Psi(-b)=\Psi(b)=0$, with $\Psi_\la|_{(-b,b)}>0$, where $b\in (0,T)$. 
  	\end{lemma}    
    By direct calculations,
    $
    \cl_+ [1]= \om - 2\vp_\la.
    $
    In particular $\om - 2\vp_\la\perp Ker[\cl_+]$. On the other hand,  $\vp_\la\perp Ker[\cl_+]$ by Lemma \ref{weaknond}. It follows that $1 \perp Ker[\cl_+]$, provided $\om\neq 0$.  Furthermore,
    $$
    \cl_+[\vp_\la]=-\vp_\la^2-a
    $$ 
    Thus, $-\vp_\la^2-a \perp Ker[\cl_+]$, so in particular $\vp_\la^2 \perp Ker[\cl_+]$. But now, we consider the function $Q(x):=\vp^2(x)-\vp(b) \vp(x)$. By construction $Q\perp Ker[\cl_+]$, so it must be that $\dpr{Q}{\Psi_\la}=0$. On the other hand, recall that $\vp_\la$ is bell-shaped, so
    $
    Q(x)=\vp(x)(\vp(x)-\vp(b))
    $
    is positive in $(-b,b)$ and it is negative in $b<|x|<1$. But this is exactly the behavior of $\Psi_\la$, in fact $Q(x) \Psi_\la(x)\geq 0$ for $-T<x<T$. Thus, $\dpr{Q}{\Psi_\la}=0$ is impossible, a contradiction. Thus, $\vp_\la$ is non-degenerate, when $\om\neq 0$. This is of course exactly the case when $a\neq \f{\la}{2}$. 
     
   \section{Orbital stability of the waves} 
   \label{sec:5} 
   We present the proof of the orbital stability, following a variation of the classical 
   T.B. Benjamin's method.  Here is a good point to discuss why the smoothness properties of the map $\la\to \vp_\la$ matters a great deal.    
   Following Benjamin's original approach,  one first considers initial data  $u_0\in H^{\f{\al}{2}}: \|u_0-\vp_\la\|_{H^{\f{\al}{2}}}<<1$, but with the additional property $P(u_0)=P(\vp_\la)=\la$. In the second step, one  removes this assumption $P(u_0)=P(\vp_\la)=\la$, that is, take  $u_0: P(u_0)\neq P(\vp_\la)$, while still close to $\vp_\la$ in $H^{\f{\al}{2}}$ metric. It has to be noted that in the original work of Benjamin, as well as many subsequent works, this second step almost automatically reduces to the first one, {\it if  the mapping $\la\to \vp_\la$ is   at least continuous  as a Banach space valued mapping\ into $L^2$}. 
   
   In some instances, for example in the classical case of a single  power non-linearity for problems posed on the line $\rone$, the function $\la\to \vp_\la$ is explicitly known by scaling arguments, and smooth by inspection,  as stated.  Virtually in all other cases, like for the waves constructed herein, scaling is not available and this becomes non-trivial. On the other hand, many authors feel that this is a natural assumption and they  explicitly take this as an assumption (and even stronger assumptions like the differentiability in spaces stronger than $L^2$), while others tacitly assume it in their arguments.  We emphasize once again that the proof presented herein   does not make any explicit assumptions beyond what is already established rigorously in   Lemma \ref{le:65}. 
   
   We start with the simpler fractional KdV case, as it presents itself with a single symmetry, namely space translation. 
      \subsection{Orbital stability for the fKdV}
   More precisely, we show 
   \begin{proposition}
   	\label{prop:lw} 
   	Let  $\vp$ be a wave, satisfying the profile equation, \eqref{30}.  Let  the conditions $(1), (2), (3)$ of Assumption  \ref{defi:15}  are satisfied and in addition the following holdс: 
   	\begin{itemize}
   		\item The operator $\cl_+=\La^\al+\om-2 \vp$ satisfies $\cl_+|_{\{\vp\}^\perp}\geq 0$. 
   		\item $\vp$ is non-degenerate, i.e. $Ker[\cl_+]=span[\vp']$.
   	\end{itemize}  
   	Then, $\vp$ is orbitally  stable. In particular, for every $\la>0, a\neq \f{\la}{2}$, the constrained minimizers $\vp_\la$ for the problem \eqref{40} are  orbitally stable.  
   \end{proposition}
   
   \begin{proof}
   Our proof proceeds  by a contradiction argument. More precisely, assuming that orbital stability does not hold, 
   there is a $\eps_0>0$ and a sequence of initial data $u_n:\lim_n\|u_n-\vp\|_{H^{\f{\al}{2}}}=0$, while for the corresponding solutions  
   \begin{equation}
   \label{orb:10}
   \sup_{0\leq t < \infty} \inf_{r\in \rone} \|u_n(t, \cdot) - \vp(\cdot-r)\|_{H^{\f{\al}{2}}} \geq \eps_0, \ n=1,2, \ldots 
   \end{equation}	
   	Note the conservation of total energy 
   	\begin{eqnarray*}
   		E[u] &=& \ch[u]+\f{\om}{2}  \cp[u]+a \cm[u]= \\
   		&=& \f{1}{2} [\int_{-T}^T |\La^{\al/2} u(t,x)|^2 dx + \om \int_{-T}^T u^2(t,x)  dx] - 
   		\f{1}{3} \int_{-T}^T u^3(t,x) + a \int_{-T}^T u(t,x) dx.
   	\end{eqnarray*}
   The profile equation, \eqref{30}, is clearly equivalent to $E'[\vp]=0$. Introduce 
   $$
   \eps_n:= |\ce(u_n(t))-\ce(\vp)|+|\cp(u_n(t))-\cp(\vp)|, 
   $$
   which is conserved in time. Note that $\lim_n \eps_n=0$, since $\lim_n \|u_n-\vp\|_{H^{\f{\al}{2}}}=0$. 
   
   	 For $0<\eps<<1$, consider a neighborhood $\mathcal{U}_{\varepsilon}$ in the set of all real-valued functions, which are close to  translations of $\vp$. More precisely, introduce 
   	 $$
   	 \mathcal{U}_{\varepsilon}=\{ u\in H_{real}^{\f{\al}{2}}[-T,T] \; \; : \; \; \inf_{r\in \rone} ||u-\varphi(\cdot -r)||_{H^{\f{\al}{2}}}<\varepsilon \}.
   	 $$
   	 By Lemma 3.2, \cite{GSS}, see also Lemma 7.7, p. 95 in \cite{Pbook},   there exists $\eps_0(\vp)>0$, so that for all $0<\eps<\eps_0(\vp)$, there  is a unique $C^1$ map $\beta : \mathcal{U}_{\varepsilon} \mapsto \mathds{R}$ such that 
   	 \begin{equation}
   	 \label{os1}
   	 \langle u(\cdot +\beta(u)), \varphi'\rangle =0, \ \ \be(\vp)=0.
   	 \end{equation}     
   	 Since we need $\eps<\min(\eps_0(\vp), \eps_0)$, take the new  $\eps_0$ to be 
   	 $\f{1}{10} \min(\eps_0, \eps_0(\vp), 1)$.   
   	 
   	 Fix for the moment $\eps<\eps_0<1$. By the continuity of the solution map (as required in  Assumption  \ref{defi:15})  and the map $\be$, 
   	 we have that there exists $t_n=t_n(\eps)>0$, so that 
   	 $\sup_{0\leq t<t_n} \|u_n(t, \cdot)-\vp\|_{H^{\f{\al}{2}}}<\f{\eps}{2}$ and $\be(u_n(t))$ is so close to $\be(\vp)=0$,  so that 
   	 $$
   	 \|\vp-\vp(\cdot-\be(u_n(t)))\|_{H^{\f{\al}{2}}}<\f{\eps}{2}. 
   	 $$
   	Consequently, for $t\in (0, t_n)$, 
   	\begin{eqnarray*}
   		& & \|u_n(t, \cdot+\be(u_n(t)))-\vp\|_{H^{\f{\al}{2}}} = \|u_n(t, \cdot)-\vp(\cdot-\be(u_n(t)))\|_{H^{\f{\al}{2}}} \leq \\
   		&\leq &  \|u_n(t, \cdot)-\vp\|_{H^{\f{\al}{2}}}+ \|\vp-\vp(\cdot-\be(u_n(t)))\|_{H^{\f{\al}{2}}}<
   		\f{\eps}{2}+ \f{\eps}{2}=\eps. 
   	\end{eqnarray*}
   	Based on this, for large $n$ and $\eps<\eps_0(\vp)$, one may define $T_n^*=T_n^*(\eps)>0$, so that 
   	$$
   	T_n^*=\sup\{\tau_0:  \sup_{0<t<\tau_0}   \|u_n(t, \cdot+\be(u_n(t)))-\vp(\cdot)\|_{H^\f{\al}{2}}<\eps\}. 
   	$$
   	The previous calculation implies $T_n^*\geq t_n$.  
   	Our goal is to show that for all sufficiently small $\eps$, there exists $N_\eps$, so that for all  $n>N_\eps$, $T_n^*=\infty$, which will provide the required contradiction with \eqref{orb:10}. 
   	 
   	 We henceforth work with $t\in (0, T_n^*)$.   
   	 Denote 
   	 $$
   	 \psi_n(t, \cdot)=u_n(t,\cdot +\beta(u_n(t)))-\varphi(\cdot)=\mu_n(t) \varphi+\eta_n(t, \cdot),  \ \ \eta_n \perp \vp.  
   	 $$
   	 Note that the definition of $T_n^*$ is equivalent with 
   	\begin{equation}
   	\label{812} 
   \sup_{0<t<T_n^*} 	\|\psi_n(t)\|_{H^\f{\al}{2}}<\eps, \ \ \limsup_{t\to T_n^*-} 	\|\psi_n(t)\|_{H^\f{\al}{2}}\geq \eps.
   	\end{equation}
   	 We have that 
   	 \begin{eqnarray*}
   	 	\cp(u_n(t)) &=&   \cp(\varphi)+ 2 \langle \varphi, \mu_n(t) \varphi+\eta_n\rangle+  \|\psi_n\|_{L^2}^2=  \cp(\varphi)+2\mu_n(t) \|\varphi\|^2+ \|\psi_n\|_{L^2}^2. 
   	 \end{eqnarray*}    
   	  It follows that 
   	  $
   	  2\mu_n \|\varphi\|^2=\cp(u_n)-\cp(\vp)- \|\psi_n\|_{L^2}^2
   	  $
   	  whence 
   	  \begin{equation}
   	  \label{orb:30} 
   	  |\mu_n|\leq \f{|\cp(u_n)-\cp(\vp)|+\|\psi_n\|_{L^2}^2}{2\|\vp\|^2}\leq C (\eps_n+\|\psi_n\|_{L^2}^2).
   	  \end{equation}
   	  Since  $E'(\vp)=0$, 
   	  Taylor expanding and various Sobolev embedding estimates yield  the formula  
   	  \begin{eqnarray*}
   	  	E(u_n(t))-E(\varphi)  &=&  E(u_n(t, \cdot+\be(u_n(t)))) -E(\varphi)= E(\vp+\psi_n(t)) - E(\vp)  = \\
   	  	&=& 
   	  	\frac{1}{2}\langle \cl_+ \psi_n(t), \psi_n(t)\rangle+O(\|\psi_n(t)\|_{H^{\f{\al}{2}}}^3)=  \\
   	  	&=& \frac{1}{2}\langle \cl_+ \eta_n(t), \eta_n(t)\rangle+ 
   	  	\f{1}{2}(\mu_n^2\dpr{\cl_+ \vp}{\vp}+ 2\mu_n \dpr{\cl_+ \vp}{\eta_n}) +O(\|\psi_n(t)\|_{H^{\f{\al}{2}}}^3). 
   	  \end{eqnarray*}
   	    By construction, $\eta_n(t)\perp \vp$. In addition, from \eqref{os1}, we have that 
   	    $$
   	    \dpr{\eta_n(t)}{\vp'}=\dpr{u_n(t,\cdot+\be(u_n(t)))-\vp-\mu_n \vp}{\vp'}=0.
   	    $$ 
   	    So, $\eta_n(t)\perp span\{\vp, \vp'\}$.  Then, by the requirements of Proposition \ref{prop:lw},  $\cl_+|_{\{\vp\}^\perp}\geq 0$. In addition, by the non-degeneracy, $Ker[\cl_+]=span[\vp']$. Thus, 
   	    $$
   	    \cl_+|_{span\{\vp, \vp'\}^\perp}\geq \ka>0.
   	    $$
   	  In particular, 
   	  \begin{equation}
   	  \label{os3}
   	  \langle \cl_+ \eta_n(t), \eta_n(t)\rangle \geq \ka \|\eta_n(t)\|_{H^{\f{\al}{2}}}^2
   	  \end{equation}
   	 Regarding the other terms, note $\|\psi_n\|^2=|\mu_n|^2 \|\vp\|^2+ \|\eta_n\|^2\geq \|\eta_n\|^2$, whence 
   	 \begin{eqnarray*}
   	  \mu_n^2+ |\mu_n| \|\eta_n\|_{L^2}\leq C(\eps_n+ \|\psi_n\|_{H^{\f{\al}{2}}}^3). 
   	 \end{eqnarray*}
   	  Plugging this information into the expression for 
   	  $E(u_n(0))-E(\varphi) = E(u_n(t))-E(\varphi)$, we arrive at 
   	  \begin{equation}
   	  \label{orb:40} 
   	  \f{\ka}{2} \|\eta_n(t)\|_{H^{\f{\al}{2}}}^2 \leq C \eps_n + C \|\psi_n(t)\|_{H^{\f{\al}{2}}}^3. 
   	  \end{equation}
   	    By the definition of $\eta_n$ and \eqref{orb:30}, we have however for $t\in (0,T_n^*)$ 
   	    \begin{equation}
   	    \label{orb:50} 
   	    \|\eta_n(t)\|_{H^{\f{\al}{2}}}\geq \|\psi_n - \mu_n \vp\|_{H^{\f{\al}{2}}}\geq \|\psi_n(t)\|_{H^{\f{\al}{2}}}-|\mu_n| \|\vp\|_{H^{\f{\al}{2}}}\geq \|\psi_n(t)\|_{H^{\f{\al}{2}}}- C (\eps_n+\|\psi_n(t)\|_{H^{\f{\al}{2}}}^2),
   	    \end{equation}
   	 where the constant $C$ appearing in the previous inequality depends on $\la, \vp$, but not on $t, n$. At this point, select $\eps$ so small that $C\eps<\f{1}{2}$. It follows that for these values of $\eps$ and $t\in (0,T_n^*)$, by \eqref{orb:50}, 
   	 $$
   	  \|\eta_n(t)\|_{H^{\f{\al}{2}}}\geq \f{1}{2} \|\psi_n(t)\|_{H^{\f{\al}{2}}}- C \eps_n. 
   	 $$
   	 Plugging this back into \eqref{orb:40}, we obtain 
   	  \begin{equation}
   	  \label{815} 
   	  \|\psi_n(t)\|_{H^{\f{\al}{2}}}^2 \leq C\eps_n+C \|\psi_n(t)\|_{H^{\f{\al}{2}}}^3.
   	  \end{equation}
   	 Again, for the new constant $C$ that appears in \eqref{815}, select $\eps$ still maybe smaller, so that    	 $C\eps<\f{1}{2}$, so that we can finally conclude from \eqref{815}, 
   	 \begin{equation}
   	 \label{817} 
   	 \|\psi_n(t)\|_{H^{\f{\al}{2}}}^2 \leq D\eps_n,
   	 \end{equation}
   	 which is valid for such small $\eps$, for all $n$ and for all $t\in (0, T_n^*)$. But this means that $T_n^*=\infty$ for all large enough $n$. Indeed, for $\eps$ small as above, take $n$ so large that $\sqrt{D\eps_n}<<\eps$, which can be done since $\lim_n \eps_n=0$. Assuming that $T_n^*<\infty$ means that 
   	 $$
   	\sqrt{D\eps_n}\geq  \limsup_{t\to T_n^*-}  \|\psi_n(t)\|_{H^{\f{\al}{2}}}\geq \eps,
   	 $$
   	 a contradiction. So, $T_n^*(\eps)=\infty$ for all large enough $n$. This is now a contradiction with \eqref{orb:10}, once we pick $\eps$ small enough ( in order to satisfy the previous two conditions and in addition  $\eps<<\eps_0$) and then $n$ large enough so that $T_n^*(\eps)=\infty$.
   \end{proof}
     \subsection{Stability for the fNLS standing  waves}
     For this part of the argument, we take $a=0$ in \eqref{456}. We have similar to Proposition \ref{prop:lw}. 
     \begin{proposition}
     	\label{prop:nls} 
     	Let  $\vp$ be a wave, satisfying   Assumption  \ref{defi:15} and the following  
     	\begin{itemize}
     		\item The operator $\cl_+=\La^\al+\om-2 \vp$ satisfies $\cl_+|_{\{\vp\}^\perp}\geq 0$. 
     		\item The operator $\cl_-=\La^\al+\om- \vp$ satisfies $\cl_-|_{\{\vp\}^\perp}\geq \ka$, for some $\ka>0$. 
     		\item $\vp$ is non-degenerate, i.e. $Ker[\cl_+]=span[\vp']$. 
     	\end{itemize}  
     	Then, $\vp$ is orbitally  stable. In particular, for every $\la>0$ and $a=0$, the solutions $e^{i \om t} \vp_{\om_\la}$ of \eqref{20}, where  $\vp_\la$ are constrained minimizers for the problem \eqref{40} are  orbitally stable. 
     \end{proposition}
     
     \begin{proof}
     		Note first that the assumptions guarantee that there exists $\ka>0$, so that 
     		\begin{equation}
     		\label{z} 
     		\cl_+|_{span\{\vp, \vp'\}^\perp}\geq \ka, \cl_-|_{\{\vp\}^\perp}\geq \ka. 
     		\end{equation} 
     		The proof then proceeds again by a contradiction, as in Proposition \ref{prop:lw}. 
     		
     		Assuming that orbital stability fails, we conclude that there exists $\eps_0>0$ and a sequence of complex-valued initial data $u_n: \lim_n\|u_n-\vp\|_{H^{\f{\al}{2}}}=0$, so that for the corresponding solutions stay away from (a translate and modulated versions of) $\vp$. That is, 
     		\begin{equation}
     		\label{orb:nls10}
     		\sup_{0\leq t < \infty} \inf_{r,  \theta\in \rone,} \|u_n(t, \cdot) - e^{i \theta} \vp(\cdot-r)\|_{H^{\f{\al}{2}}}\geq \eps_0.
     		\end{equation}
     		Consider the set, for small enough $\eps$
     		$$
     		\cu_\eps:=\{u=v+i w: v, w\in H^{\f{\al}{2}}_{real}[-T,T]: \inf_r[\|v-\vp(\cdot-r)\|_{H^{\f{\al}{2}}}+ \|w\|_{H^{\f{\al}{2}}}]<\eps \}, 
     		$$
     		together with the well-defined  map $\be: \cu_\eps\to \rone$, so that 
     		\begin{equation}
     		\label{os1:nls} 
     		\dpr{v(\cdot+\be(v))}{\vp'}=0. 
     		\end{equation}
     	Letting again $E(u)=\ch(u)+\f{\om}{2} \cp(u)$ and $\eps_n:=|\ch(u_n(t))-\ch(\vp)|+|\cp(u_n(t))-\cp(\vp)|$, we observe again that $\eps_n$ is conserved and $\lim_n \eps_n=0$, since $\lim_n\|u_n-\vp\|_{H^{\f{\al}{2}}}=0$. Also, 
     	$$
     	E'[\vp]=\ch'[\vp]+\f{\om}{2} \cp'[\vp]=0.
     	$$
     We now define the appropriate translation and  modulation parameters.  The translation parameter is simply as before, $r_n(t):=\be(v_n(t))$, while the modulation parameter $\theta_n(t)$ is determined from the relation 
     \begin{equation}
     \label{600} 
     \dpr{w_n(t, \cdot+\be(v_n(t)))}{\vp}=\sin(\theta_n(t)) \|\vp\|_{L^2}^2,
     \end{equation}	
     Note that while $(v,w)\in \cu_\eps$, the expression on the left hand side of \eqref{600} is $O(\eps)$, so $\theta_n(t)$ is taken to be the unique small solution of \eqref{600}. More generally, under the {\it a priori} assumption that $u_n=v_n+i w_n$ belongs to the set $\cu_\eps$, which we will eventually uphold for all times $t$ under consideration, it follows   that both $r_n(t)=O(\eps), \theta_n(t)=O(\eps)$ are uniquely determined. 
  
     	Next,  fix small enough $\eps>0$, so that the map $\be:\cu_\eps\to \rone$ is well defined and   \eqref{os1:nls} holds.  By the continuity of the solution map and the $C^1$ property of the map $\be$,   there exists $t_n=t_n(\eps)>0$, so that 
     	$\sup_{t\in (0, t_n)} \|u_n(t, \cdot)-\vp\|<\eps$.  In particular,  
     	$$
     		\|v_n(t, \cdot)-\vp\|\leq \|u_n(t, \cdot)-\vp\|<\eps,
     		$$
     	whence $\be(v_n(t))$ is $O(\eps)$ close to $\be(\vp)=0$ and $\theta_n(t)=O(\eps)$. Thus,  
     	$$
     	|e^{i \theta_n(t)} -1| \|\vp\|_{H^{\f{\al}{2}} }<C_0\eps, \|\vp-\vp(\cdot-\be(v_n(t)))\|_{H^{\f{\al}{2}} }<C_0\eps,
     	$$
     	for some constant $C_0=C_0(\vp)$.   Thus, for $t\in (0, t_n)$, 
     	 \begin{eqnarray*}
     	 	\|u_n(t,\cdot+\be(v_n(t))- e^{i \theta_n(t)} \vp\|_{H^{\f{\al}{2}} } &\leq & \|u_n(t, \cdot)-\vp\|_{H^{\f{\al}{2}} }+ 
     	 	\|\vp-\vp(\cdot-\be(v_n(t)))\|_{H^{\f{\al}{2}} }+\\
     	 	&+& |e^{i \theta_n(t)} -1| \|\vp\|_{H^{\f{\al}{2}} }\leq (2C_0+1) \eps. 
     	 \end{eqnarray*}
     	Define 
     	$$
     	T_n^*=T_n^*(\eps):=\sup\{\tau_0:  \sup_{0<\tau<\tau_0}   \|u_n(\tau, \cdot+\be(u_n(\tau)))- 
     	e^{i \theta_n(\tau)}\vp(\cdot)\|_{H^\f{\al}{2}}<2(2C_0+1)\eps\},
     	$$
     	so that the previous calculation implies $T_n^*\geq t_n>0$. We will show that for all small enough $\eps$, there exists $N=N_\eps$, so that for all $n>N_\eps$, $T_n^*=\infty$. This is in  a contradiction with \eqref{orb:nls10}, by taking $\eps<<\eps_0$ and correspondingly  large $N_\eps$.   
     	
     	Write for $t\in (0, T_n^*)$, 
     	\begin{eqnarray*}
     		\psi_n(t) &:=& u_n(t,\cdot+\be(v_n(t))- e^{i \theta_n(t)} \vp=  \\
     		&=& 
     		v_n(t, \cdot+\be(v_n(t)))-\cos(\theta_n(t)) \vp+i (w_n(t, \cdot+\be(v_n(t)))-\sin(\theta_n(t)) \vp)
     	\end{eqnarray*}
     	Note that $t\in (0, T_n^*)$ implies  $\|\psi_n(t)\|_{H^{\f{\al}{2}}}<2(2C_0+1)\eps$.  Viewing $\psi_n(t)$  as a vector in the real and imaginary parts, we project over the vector $\left(\begin{array}{c} 
     	\vp \\ 0 
     	\end{array} \right)$ and its orthogonal complement, whence we obtain the representation  
     	\begin{equation}
     	\label{610} 
     	\left(\begin{array}{c} 
     	v_n(t, \cdot+\be(v_n(t)))-\cos(\theta_n(t)) \vp \\
     	w_n(t, \cdot+\be(v_n(t)))-\sin(\theta_n(t)) \vp 
     	\end{array} \right)=\mu_n(t) \left(\begin{array}{c} 
     	\vp \\ 0 
     	\end{array} \right)+\left(\begin{array}{c} 
     	\eta_n(t) \\ \zeta_n(t)
     	\end{array} \right),   \left(\begin{array}{c} 
     	\eta_n(t) \\ \zeta_n(t)
     	\end{array} \right) \perp \left(\begin{array}{c} 
     	\vp \\ 0 
     	\end{array} \right). 
     	\end{equation} 
     	By the construction of $\beta(v_n(t))$, we have by \eqref{os1:nls} that  $\dpr{v_n(\cdot+\beta(v_n(t)))}{\vp'(\cdot)}=0$, so taking dot product of the first equation in \eqref{610} with $\vp'$ yields 
     	$
     	\dpr{\eta_n(t)}{\vp'}=0,
     	$
     	or $\eta_n(t)\perp \vp'$. 
     	
  Furthermore, the condition 
  $\left(\begin{array}{c} 
   \eta_n(t) \\ \zeta_n(t)
   \end{array} \right) \perp \left(\begin{array}{c} 
   \vp \\ 0 
   \end{array} \right)$ is nothing,  but $\eta_n(t)\perp \vp$, so $\eta_n(t)\perp \vp, \vp'$. 
   In addition,  the choice of $\theta_n$ in \eqref{600} is equivalent to $w_n(t, \cdot+\be(v_n(t)))-\sin(\theta_n(t)) \vp \perp \vp$, which translates to exactly  
   $\zeta_n(t)  \perp \vp$. 
   It is then clear that 
   \begin{equation}
   \label{715} 
   	\cp(u_n(t)) = \cp(u_n(t,\cdot+\be(v_n(t))))= \cp(\vp)+ 2\mu_n(t) \cos(\theta_n(t)) \|\vp\|_{L^2}^2 +  \|\psi_n(t)\|_{L^2}^2.  
   \end{equation}
     	Taking into account $\theta_n(t)=O(\eps)$ (so $\cos(\theta_n(t))=1+O(\eps^2)$), we obtain from \eqref{715}, 
     	\begin{equation}
     	\label{620} 
     	|\mu_n(t)|\leq  \f{|\cp(u_n(t))-\cp(\vp)|+ \|\psi_n(t)\|_{L^2}^2}{2 \cos(\theta_n(t)) \|\vp\|_{L^2}^2} 
     	\leq C(\eps_n+\|\psi_n(t)\|_{L^2}^2)\leq C(\eps_n+\eps^2),
     	\end{equation}
     	where in the last inequality, we have used that $t\in (0, T_n^*)$. Note in addition, that taking $L^2$ norms in \eqref{610} and using the orthogonality relations yields 
     	\begin{equation}
     	\label{621} 
     	|\mu_n(t)|^2+\|\zeta_n(t)\|^2+ \|\eta_n\|^2=\|\psi_n(t)\|^2\leq C\eps^2,
\end{equation} 
     	Now, 
     	\begin{eqnarray*}
     & & 		E[u_n(t)]-E[\vp] = E[u_n(t,\cdot+\be(v_n(t)))]-E[ \vp]=
     		E[e^{i \theta_n(t)} \vp+\psi_n]-E[ \vp]=\\
     		&=& E[(\cos(\theta_n(t)) \vp+\mu_n \vp+\eta_n)+ i(\sin(\theta_n(t)) \vp+\zeta_n)]-E[ \vp].
     	\end{eqnarray*}
 Note 	
     		\begin{eqnarray*}
     			 & &  |(\cos(\theta_n ) \vp+\mu_n \vp+\eta_n)+ i(\sin(\theta_n) \vp+\zeta_n)|^2=
     			\vp^2+2 \cos(\theta_n) \vp(\mu_n \vp+\eta_n)+(\mu_n \vp+\eta_n)^2+\\
     			&+&2\sin(\theta_n) \vp\zeta_n +\zeta_n^2= \vp^2+2  \vp(\mu_n \vp+\eta_n)+(\mu_n \vp+\eta_n)^2+2\sin(\theta_n) \vp\zeta_n +\zeta_n^2+O(\eps^2(|\zeta_n|+|\eta_n|)). 
     		\end{eqnarray*}
     	 where we have used  $\cos(\theta_n(t))=1+O(\eps^2)$. By the relations \eqref{620}, \eqref{621},  $\mu_n^2+|\mu_n|\eta_n|\leq C(\eps_n+\eps^3)$. It follows that 
     			\begin{eqnarray*}
     				& &   E[(\cos(\theta_n(t)) \vp+\mu_n \vp+\eta_n)+ i(\sin(\theta_n(t)) \vp+\zeta_n)]-E[ \vp] = 
     				\dpr{\La^{\al} \vp}{\mu_n \vp+\eta_n}+  \f{1}{2}  \dpr{\La^{\al}  \eta_n }{\eta_n}+\\ 
     				&+&  \sin(\theta_n)\dpr{\La^\al \vp}{\zeta_n}+  \f{1}{2}  \dpr{\La^{\al}  \zeta_n }{\zeta_n}+\om\dpr{\vp}{\mu_n\vp+ \eta_n+\sin(\theta_n) \zeta_n}+\f{\om}{2}(\dpr{\eta_n}{\eta_n}+\dpr{\zeta_n}{\zeta_n})\\
     				&-& \mu_n\dpr{\vp^2}{\vp} - \dpr{\vp^2}{\eta_n} -\dpr{\vp}{\eta_n^2}-\sin(\theta_n)\dpr{\vp^2}{\zeta_n}-\f{1}{2}\dpr{\vp}{\zeta_n^2} + O(\eps_n+\eps^3). 
     			\end{eqnarray*}
     	By the profile equation, $\La^\al\vp+\om \vp- \vp^2=0$, we can simplify the expression above 
     	\begin{equation}
     	\label{718}
     	E[u_n(t)]-E[\vp]\geq \f{1}{2}[\dpr{\cl_+ \eta_n(t)}{\eta_n(t)}+\dpr{\cl_- \zeta_n(t)}{\zeta_n(t)}] - C(\eps^3+|\eps_n|).
     	\end{equation}
     	As we have pointed out, $\eta_n(t)\perp span\{\vp, \vp'\}$, $\zeta_n(t)\perp\vp$, so \eqref{z}  above   implies 
     	$$
     	\dpr{\cl_+ \eta_n}{\eta_n}+\dpr{\cl_- \zeta_n}{\zeta_n}\geq \ka[\|\eta_n\|_{H^{\f{\al}{2}}}^2+\|\zeta_n\|_{H^{\f{\al}{2}}}^2]
     	$$ 
     	We conclude, by taking into account $|E[u_n(t)]-E[\vp] |\leq \eps_n$, and $t\in (0,T_n^*)$
    \begin{equation}
    \label{812} 
    	\ka[\|\eta_n(t)\|_{H^{\f{\al}{2}}}^2+\|\zeta_n(t)\|_{H^{\f{\al}{2}}}^2]\leq C(\eps_n+\eps^3).
    \end{equation}
     	This implies however that for all $t\in (0, T_n^*)$, we have (again, using \eqref{620} for $\mu_n(t)$), 
     \begin{equation}
     \label{817} 
     	\|\psi_n(t)\|_{H^{\f{\al}{2}}}\leq C \sqrt{\eps_n}+ C\eps^{\f{3}{2}}
\end{equation}
     	But then, for sufficently small $\eps$ and for large enough $n$, we must have $T_n^*=\infty$.  Indeed, otherwise 
     	$$
     	C \sqrt{\eps_n}+ C\eps^{\f{3}{2}}\geq \limsup_{t\to T_n^*-} \|\psi_n(t)\|_{H^{\f{\al}{2}}}\geq C_1 \eps. 
     	$$
     	Such an inequality clearly will not hold by selecting $\eps: C\sqrt{\eps}<\f{C_1}{2}$ and then $n$ so large that $\sqrt{\eps_n}<<\eps$, which can be done since $\lim_n \eps_n=0$.  Thus a contradiction is reached and the waves are orbitally stable. 
     \end{proof}

\end{document}